\documentclass[11pt]{article}
\usepackage{etex}
\usepackage{geometry}
\usepackage{amsmath}
\usepackage{amsfonts}
\usepackage{amssymb}
\usepackage{amsthm}
\usepackage{mathtools}
\usepackage{dsfont}
\usepackage{natbib}
\setcitestyle{aysep={}}
\usepackage{hyperref}
\usepackage{breakcites}
\usepackage{cleveref}
\usepackage{booktabs}
\usepackage{url}
\usepackage[bb=boondox]{mathalfa}
\usepackage{fancybox}
\usepackage{multirow}
\usepackage{soul}

\usepackage[utf8]{inputenc}
\usepackage[T1]{fontenc}
\usepackage{lmodern}

\usepackage[dvipsnames]{xcolor}
\usepackage{tikz}
\usetikzlibrary{arrows,shapes,positioning}
\usepackage{tikz-cd}
\usepackage[font={small,it}]{caption}
\usepackage[ruled,noline]{algorithm2e}
\usepackage[normalem]{ulem}

\usepackage[colorinlistoftodos,prependcaption]{todonotes}

\newtheorem{theorem}{Theorem}
\newtheorem{lemma}{Lemma}

\theoremstyle{definition}
\newtheorem{definition}{Definition}

\theoremstyle{remark}
\newtheorem{remark}{Remark}
\newtheorem{example}{Example}

\newcommand{\Mod}[1]{\ (\mathrm{mod}\ #1)}

\newcommand{\N}{\mathcal{N}}
\newcommand{\M}{\mathcal{M}}
\renewcommand{\S}{\mathcal{S}}
\newcommand{\C}{\mathcal{C}}

\newcommand{\R}{\mathcal{R}}

\newcommand{\tN}{\tilde{\mathcal{N}}}
\newcommand{\tM}{\tilde{\mathcal{M}}}
\newcommand{\tS}{\tilde{\mathcal{S}}}
\newcommand{\tC}{\tilde{\mathcal{C}}}
\newcommand{\tR}{\tilde{\mathcal{R}}}

\newcommand{\mA}{\mathbf{A}}
\newcommand{\mY}{\mathbf{Y}}
\newcommand{\mI}{\mathbf{I}}
\newcommand{\mP}{\mathbf{P}}
\newcommand{\mU}{\mathbf{U}}
\newcommand{\mV}{\mathbf{V}}
\newcommand{\mGamma}{\mathbf{\Gamma}}
\newcommand{\mUpsilon}{\mathbf{\Upsilon}}
\newcommand{\vx}{\mathbf{x}}

\newcommand{\AI}{XD}
\newcommand{\AII}{X}
\newcommand{\AIII}{XT}
\newcommand{\AIV}{X_p}
\newcommand{\AV}{Y}
\newcommand{\AVI}{X_pY}
\newcommand{\AVII}{Y_p}
\newcommand{\AVIII}{XTY_p}
\newcommand{\AIX}{XDY_p}

\title{Network translation and steady state properties\\of chemical reaction systems}
\author{Elisa Tonello \\
        School of Mathematical Sciences \\
        University of Nottingham \\
        Nottingham, NG7 2RD \\
        \href{mailto:elisa.tonello@nottingham.ac.uk}{elisa.tonello@nottingham.ac.uk}
        \and
        Matthew D. Johnston \\
        Department of Mathematics \\
        San Jose State University \\
        San Jose, CA, 95192 USA \\
        \href{mailto:matthew.johnston@sjsu.edu}{matthew.johnston@sjsu.edu}
}
\date{}

\begin{document}

\maketitle

\begin{abstract}
Network translation has recently been used to establish steady state properties of mass action systems by corresponding the given system to a generalized one which is either dynamically or steady state equivalent. In this work we further use network translation to identify network structures which give rise to the well-studied property of absolute concentration robustness in the corresponding mass action systems. In addition to establishing the capacity for absolute concentration robustness, we show that network translation can often provide a method for deriving the steady state value of the robust species. We furthermore present a MILP algorithm for the identification of translated chemical reaction networks that improves on previous approaches, allowing for easier application of the theory.
\end{abstract}

\tableofcontents

\section{Introduction}

Biochemical systems are often modeled as graphs of chemical reactions that convert reactant species to product species.
Under some reasonable assumptions on the kinetics, these networks are studied mathematically via systems of ordinary differential equations (ODEs).
A particularly common kinetic assumption is that of \emph{mass action}, whereby the rate of each reaction takes the form of a monomial with support on the reactant species. The study of the resulting mass action systems is very common in systems biology, but is made challenging by high dimensionality, abundant nonlinearities, and parameter uncertainty.

One topic of significant recent interest within the study of mass action systems has been \emph{absolute concentration robustness} (ACR). Robustness in key substrate concentrations in fluctuating environments has been observed experimentally in \cite{Alon1999} and \cite{Shinar2007,Shinar2009-2} and is considered an important feature of homeostasis in biochemical regulatory systems. Mathematically, a mass action system is said to have ACR in species $A$ if the concentration $x_A$ takes the same value at every positive steady state of the corresponding mass action system. \cite{ShinarFeinberg2010Robustness} established a simple structural criterion for the identification of robustness in networks of deficiency one. \cite{neigenfind2013relation} suggested combining invariant flux ratios and invariant complex ratios for the identification of robustness, while a constructive linear method for generating polynomial invariants, which can reveal ACR, was introduced by \cite{karp2012complex}. ACR networks have also been studied recently in a stochastic modeling framework \citep{A-E-J,Enciso2016,Anderson2016}.

In this paper, we present a graphical approach for affirming ACR through novel application of the theories of generalized mass action systems and network translation \citep{muller2012generalized,johnston2014translated}. Like a traditional reaction network, a generalized network consists of vertices corresponding to complexes and arrows corresponding to reactions. However, in a generalized network two complexes are assigned to each vertex, a \emph{stoichiometric complex} (unbracketed) and a \emph{kinetic complex} (bracketed), and these complexes govern different portions of the associated system of ODEs. The process of network translation relates a network to a generalized one and can be visualized by adding a linear combination of species, called a \emph{translation complex}, to the left-hand and right-hand sides of individual reactions.
Generalized mass action representations are also natural in the context of population dynamics and epidemiology. For example, the transition represented as $S \to I$ in a Susceptible-Infected model typically occurs at non-mass-action rate $\beta x_S x_I$ (for further graphical examples, see \cite{brauer2012}).

To illustrate how network translation may reveal ACR, consider the following network:
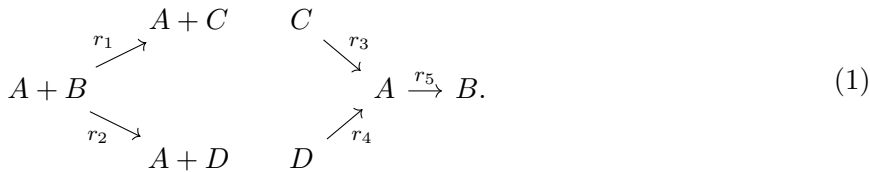
\begin{equation}\label{eq:proper_acr_intro}
      \begin{tikzcd}[row sep=small, column sep=small]
          & A + C & C \arrow[rd, "r_3"] & & \\
         A + B \arrow[ru, "r_1"] \arrow[rd, "r_2"'] & & & A \arrow[r, "r_5"] & B. \\
        & A + D & D \arrow[ru, "r_4"'] & &
      \end{tikzcd}
\end{equation}
The mass action system associated to \eqref{eq:proper_acr_intro} can be shown to have ACR in $B$ through analysis of the steady state equations, or of the complex-linear invariants~\citep{karp2012complex}; however, the structural results of \cite{ShinarFeinberg2010Robustness} do not directly identify ACR in any species.
A different graphical representation of the same network can on the other hand reveal the property: consider adding the translation complex $A$ to the reactant and product complex of reactions $r_3$, $r_4$, and $r_5$. This \emph{translation scheme} generates the following generalized network:
    
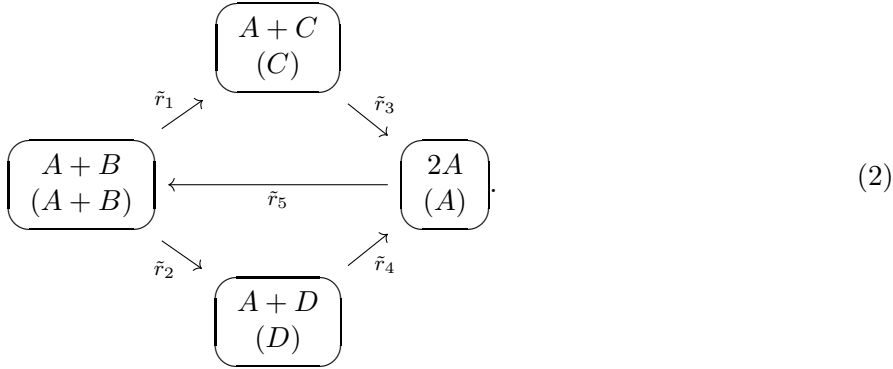
\begin{equation}\label{eq:proper_acr_intro_translated}
      \begin{tikzcd}[row sep=small, column sep=small]
        & & \mbox{\ovalbox{$\begin{array}{c} A+C \\ (C) \end{array}$}} \arrow[rd, "\tilde{r}_3"] & & & \\
& \mbox{\ovalbox{$\begin{array}{c} A+B \\ (A+B) \end{array}$}} \arrow[ru, "\tilde{r}_1"] \arrow[rd, "\tilde{r}_2"'] & & \mbox{\ovalbox{$\begin{array}{c} 2A \\ (A) \end{array}$}}. \arrow[ll, "\tilde{r}_5"] \\
        & & \mbox{\ovalbox{$\begin{array}{c} A+D \\ (D) \end{array}$}} \arrow[ru, "\tilde{r}_4"'] & & &
      \end{tikzcd}
    \end{equation}
\noindent When generalized networks arise from network translation, as in \eqref{eq:proper_acr_intro_translated}, the stoichiometric complexes correspond to the original complexes adjusted by the translation complex (in this case, we have added $A$ to $r_3$, $r_4$, and $r_5$), and the kinetic complexes correspond to the original source complexes \citep{johnston2014translated}.
The first result of this paper (Theorem \ref{thm:robustness_translation}) guarantees that every pair of source complexes in \eqref{eq:proper_acr_intro} have a robust ratio because they are translated to a common linkage class in \eqref{eq:proper_acr_intro_translated}.
We can establish ACR of $x_B$ for the mass action system corresponding to \eqref{eq:proper_acr_intro} by noting that $A+B$ and $A$ appear as kinetic complexes in the same linkage class of \eqref{eq:proper_acr_intro_translated}.
 Theorem \ref{thm:resolvability_translation} allows in addition the explicit computation of the steady state value for $x_B$.

The present work builds on the theory developed in \cite{johnston2014translated}. We show how network translation can be used to establish ACR,
  and clarify how robustness properties can be interpreted in terms of the graphical tool given by \emph{tree constants},
  outside of the standard deficiency zero-weakly reversible setting.
In addition, we present a mixed-integer linear program (MILP) for constructing network translations.
For instance, we computationally determine the translation scheme required to correspond \eqref{eq:proper_acr_intro} to \eqref{eq:proper_acr_intro_translated}.
The algorithm realizes the objective first stated in \cite{johnston2014translated} of using a network's elementary modes to build a network translation.
In contrast to the MILP algorithm presented in \cite{johnston2015computational},
our algorithm does not depend on a priori knowledge of either the stoichiometric complexes in the translation or the rate constants of each reaction.

The paper is structured as follows.
In Section \ref{sec:defs}, we present the relevant background about chemical reaction networks and robustness. In Section~\ref{sec:translated}, we introduce translated chemical reaction networks
and illustrate how these can be used to correspond mass action systems to generalized mass action systems with lower deficiency, as well as for the identification of robustness properties.
In Section~\ref{sec:translation_algo}, we present a new computational approach to the identification of translated reaction networks. In Section \ref{sec:conclusions}, we summarize our results and present some avenues for future research.

\section{Background}\label{sec:defs}

In this section we introduce the relevant background on chemical reaction networks and the corresponding dynamical modeling framework. The terminology used in this paper is inspired significantly by \emph{Chemical Reaction Network Theory} (CRNT), a framework which has proved very effective in recent years in analyzing the steady states of mass action systems
\citep{feinberg1972complex,Feinberg1979lectures,feinberg1987chemical,feinberg1988chemical,feinberg1995existence,feinberg1995multiple,horn1972necessary,horn1972general,ShinarFeinberg2010Robustness}.

\subsection{Chemical reaction networks}
\label{sec:crn}

The set of species of a chemical reaction network is denoted $\S = \{ X_1, \ldots, X_n \}$.
A finite integer linear combination of the species of the form $y = \sum_{i=1}^n y_{i} X_i$ is called a \emph{complex},
and the non-negative integers $y_{i}$ are called \emph{stoichiometric coefficients}.
The set of complexes of the chemical reaction network is denoted $\C = \{y^1, \dots, y^c\}$, where $c$ is the cardinality of $\C$.
It will also sometimes be useful to identify a complex $y$ with the vector of its
stoichiometric coefficients, e.g. $y = {(y_1, \ldots, y_n)} \in \mathbb{N}_0^n$.

A \emph{reaction} $r$ is an ordered pair of complexes ($y$, $y'$) $\in \mathcal{C} \times \mathcal{C}$, with $y \neq y'$, and is denoted as $r: y \rightarrow y'$ or $y \xrightarrow{r} y'$. The complex $y$ is called the \emph{reactant} of the reaction while $y'$ the \emph{product} of the reaction. The finite set of reactions of a network is denoted $\R = \{ r_1, \ldots, r_m \}$.
We will denote a chemical reaction network as a triple $\N = (\S, \C, \R)$ and furthermore assume that every species appears in at least one complex, and every complex appears in at least one reaction.

The vector $y' - y \in \mathbb{Z}^n$
 is called the \emph{reaction vector} of the reaction $y \to y'$. The \emph{stoichiometric subspace} of a chemical reaction network is defined as
\[S = \mbox{span}\{y' - y \; |\ r: y \rightarrow y',\ r \in \R\}.\]
 The dimension $s = \mbox{dim}(S)$ is called the $\emph{rank}$ of the reaction network.

The pair $(\C, \R)$ identifies a directed graph which we will call the \emph{graph of complexes} of the chemical reaction network.
The strongly connected components of this graph are the \emph{strong linkage classes} of the reaction network.
The \emph{linkage classes} $\mathcal{L}_\theta$, $\theta = 1, \dots, l$, of the reaction network are the connected components of the undirected graph obtained from the graph of complexes by ignoring the direction of the reactions.
A chemical reaction network is said to be \emph{weakly reversible} if its linkage classes coincide with its strong linkage classes.
A strong linkage class is said to be terminal if it admits no outgoing edges. The \emph{deficiency} of a chemical reaction network is a nonnegative integer defined as $\delta = c - l - s$. The deficiency has been the study of significant research of the past 45 years \citep{feinberg1972complex,feinberg1987chemical,horn1972necessary,ShinarFeinberg2010Robustness}.

\begin{example}
\label{ex:example1}
Consider the following network, taken from \cite{ShinarFeinberg2010Robustness}:
\begin{equation}\label{eq:network1}
  \begin{tikzcd}
    A+B \arrow[r, "r_1"] & 2B, \\[-0.2in]
    B \arrow[r, "r_2"] & A.
  \end{tikzcd}
\end{equation}
This network has the species set $\S = \{ A, B\}$, the complex set $\C = \{
A+B, 2B, B, A \}$, and the reaction set $\R = \{ A + B \xrightarrow{} 2B, B
\xrightarrow{} A \}$.
The linkage classes are $\{A + B, 2B\}$ and $\{B, A \}$, while the strong
linkage classes are $\{A+B \}$, $\{ 2B\}$, $\{ B \}$, and $\{A \}$.
The network is not weakly reversible and only the strong linkage classes $\{2B\}$ and $\{A \}$ are terminal.
If we label $y^1 = A+B$, $y^2 = 2B$, $y^3 = B$, and $y^4 = A$, then the reaction vectors are $y^{2} - y^{1} = (0,2) - (1,1) = (-1,1)$ and $y^{4} - y^{3} = (1,-1)$ so that $S = \mbox{span}\{ (-1,1), (1,-1) \} = \mbox{span} \{ (-1,1) \}$. It follows that $s = 1$ so that this is a rank $1$ network. The deficiency is $\delta = c - l - s = 4 - 2 - 1 = 1$.
\end{example}

\begin{example}
\label{ex:envz}
Consider the following network, which was first presented as S.60 in \cite{ShinarFeinberg2010Robustness} as a model of the EnvZ/OmpR osmoregularity signaling pathway in \emph{escherichia coli}:
    
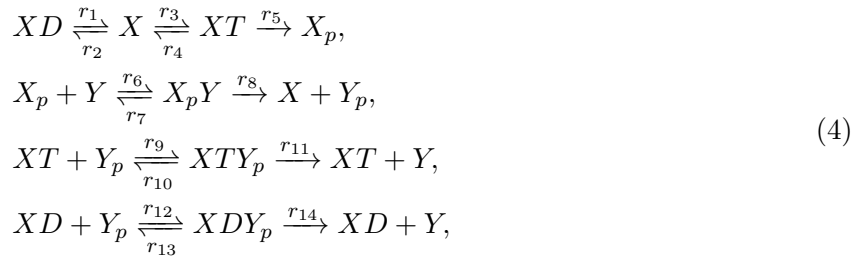
\begin{equation}\label{eq:envz_ompr}
        \begin{aligned}
          & \AI \xrightleftharpoons[r_2]{r_1} \AII \xrightleftharpoons[r_4]{r_3} \AIII \xrightarrow{r_5} \AIV, \\
          & \AIV + \AV \xrightleftharpoons[r_7]{r_6} \AVI \xrightarrow{r_8} \AII + \AVII, \\
          & \AIII + \AVII \xrightleftharpoons[r_{10}]{r_9} \AVIII \xrightarrow{r_{11}} \AIII + \AV, \\
          & \AI + \AVII \xrightleftharpoons[r_{13}]{r_{12}} \AIX \xrightarrow{r_{14}} \AI + \AV,
        \end{aligned}
    \end{equation}
    where $X =$ EnvZ, $Y =$ OmpR, $D =$ ADP, $T =$ ATP, and $p =$ phosphate group.
 This network has the species set $\mathcal{S} = \{ XD, X, XT, X_p, Y, X_pY,
    Y_p, XTY_p, XDY_p\}$, so that $n=9$, and the complex set $\mathcal{C} = \{
    XD, X, XT, X_p, X_p + T, X_pY, X+Y_p, XT+Y_p, XTY_p, XT + Y, XD+Y_p, XDY_p,
    XD + Y \}$, so that $c=13$. The network has $4$ linkage classes, $8$ strong
    linkage classes, and $4$ terminal strong linkage classes. The network is not
    weakly reversible and the rank is $s =7$ so that the deficiency is $\delta =
    c - l - s = 13 - 4 - 7 = 2$.
\end{example}

\subsection{Mass action systems}

A common kinetic assumption is that of mass action, whereby the rate of a reaction is proportional to the product of the reactant species concentrations.
For example, a reaction of the form $A + B \xrightarrow{} \cdots$ would have
rate $k x_A x_B$ where $k \in \mathbb{R}_{> 0}$ is the \emph{rate constant} of the reaction and $x_A$ and $x_B$ denote the concentrations of $A$ and $B$, respectively.
Notably, for this choice of kinetics, the source complex $y = \sum_{i=1}^n y_{i} X_i$ of each reaction is assigned a monomial $\vx^y = x_1^{y_1} \cdots x_n^{y_n}$ where $x_i$ is the concentration of $X_i$.
We call this the \emph{mass action monomial} associated to $y$.

Under \emph{mass action kinetics}, the dynamics of a chemical reaction network $\N$ is governed by the system of polynomial ODEs
\begin{equation}\label{eq:dynamics}
  \frac{d\vx}{dt} =\sum_{r:y \to y' \in \mathcal{R}}  k_{r} (y' - y) \; \vx^y,
\end{equation}
where $k_{r} \in \mathbb{R}_{>0}$ is the \emph{rate constant} associated to the reaction $r \in \mathcal{R}$. Given an enumeration $\mathcal{R} = \{ r_1, \ldots, r_m\}$ of the reaction set, we write $k = (k_1, \dots, k_m)$ for an assignment of rate constants and
call the tuple $\M = (\S, \C, \R, k)$ a \emph{mass action system}. A consequence of \eqref{eq:dynamics} is that $\mathbf{x}'(t) \in S$ for all $t \geq 0$. Solutions of \eqref{eq:dynamics} are consequently restricted to \emph{stoichiometric compatibility classes} $(S+ \vx_0) \cap \mathbb{R}_{\geq 0}^n$ for each initial state $\vx_0 \in \mathbb{R}_{\geq 0}^n$.

The system of ODEs~\eqref{eq:dynamics} of the reaction network has the following useful reformulation:
\begin{equation}\label{eq:dyn_Ak}
  \frac{d\vx}{dt} = \mY \mA_\kappa \Psi(\vx),
\end{equation}
where: (1) $\mY$ is the $n \times c$ \emph{complex matrix} with columns the vectors of stoichiometric coefficients of $y^1, \dots, y^c$; (2) $\mA_\kappa$ is a $c \times c$ \emph{kinetic matrix} with entries ${(\mA_{\kappa})}_{ij} = k_r$ if $r : y^j \rightarrow y^i$ for $i \neq j$, and columns that sum to zero; and (3) $\Psi(\vx) = (\vx^{y^1}, \dots, \vx^{y^c})^t$ is the vector of mass action monomials.

Consider the following example.

\begin{example}
\label{ex:example3}
Reconsider \emph{Network 1} given in Example~\ref{ex:example1}. Under mass action kinetics, we associate to $r_1$ the rate $k_1 x_A x_B$ and to $r_2$ the rate $k_2 x_B$. The corresponding mass action system \eqref{eq:dynamics} is given by
\begin{equation}
\label{eq:n1}
\begin{split}
\frac{dx_A}{dt} & = -k_{1} x_A x_B + k_{2} x_B, \\
\frac{dx_B}{dt} & = k_{1} x_A x_B - k_{2} x_B.
\end{split}
\end{equation}
\end{example}

\subsection{Absolute concentration robustness}

\noindent Significant research has been conducted on the study of the steady states, and in particular positive steady states, of the system of ODEs \eqref{eq:dynamics} \citep{horn1972general,craciun2009toric,millan2012chemical}.
We will be interested in the following steady state properties:
\begin{definition}
\label{def:acr}
A mass action system $\M = (\S, \C, \R, k)$ is said to have:
\begin{enumerate}
\item
a \emph{robust ratio} between complexes $y$ and $y'$ if the ratio $\mathbf{x}^{y}/\mathbf{x}^{y'}$ takes the same value at every positive steady state $\mathbf{x} \in \mathbb{R}_{> 0 }^n$ of \eqref{eq:dynamics}.
\item
\emph{absolute concentration robustness} in species $A$ if $x_A$ takes the same value at every positive steady state $\mathbf{x} \in \mathbb{R}_{> 0 }^n$ of \eqref{eq:dynamics}.
\end{enumerate}
\end{definition}
\noindent Note that, as a special case, if two complexes $y$ and $y'$ have a robust ratio and differ only in a single species $A$, then ACR is guaranteed for $x_A$ \citep{ShinarFeinberg2010Robustness}. It is useful to introduce the \emph{robustness space} $R \subseteq \mathbb{R}^n$ of a mass action system $\M$ as the subspace given by
\begin{equation}
\label{eq:robustness_space}
R = \mathrm{span}\{y' - y \; | \; y \mbox{ and } y' \mbox{ have a robust ratio}\}.
\end{equation}
This definition allows us to restate an observation made by \citet[Proposition $S4.1$]{ShinarFeinberg2010Robustness} and \citet[Lemma $11$]{neigenfind2013relation} in slightly different forms.

\begin{lemma}
\label{lemma1}
  Consider a chemical reaction network $\mathcal{N}$ with corresponding mass action system $\M$ and robustness space $R$.
  If $y, y' \in \mathbb{Z}_{\geq 0}^n$ and $y' - y \in R$, then $y$ and $y'$ have a robust ratio.
  In particular, if $\mathbf{e}^i \in R$, where $\mathbf{e}^i \in \mathbb{Z}^n$ is the vector with $\mathbf{e}^i_i = 1$ and $\mathbf{e}^i_j = 0$ for $j \neq i$,
   then $\M$ has ACR in $X_i$.
\end{lemma}
\begin{proof}
Let $\mathcal{C}_{\nu} \subseteq \mathcal{C}$, $\nu = 1, \ldots, \mu$, denote sets of complexes which have robust ratios, i.e. $y^i, y^j \in \mathcal{C}_{\nu}$ implies $y^i$ and $y^j$ have a robust ratio. Then, if $y'-y \in R$, we have
  \begin{equation*}
    y'-y=\sum_{\substack{y^i,y^j \in \C_\nu\\{\nu=1,\dots,\mu}}} \alpha_{ij} (y^j-y^i),
  \end{equation*}
  for some $\alpha_{ij} \in \mathbb{R}$, and $\frac{\vx^{y^j}}{\vx^{y^i}}=\beta_{ij}$ at each positive steady state, for some $\beta_{ij}\in\mathbb{R}_{>0}$.
  Then at each positive steady state we have
  \begin{equation*}
    \frac{\vx^{y'}}{\vx^y} = \prod_{\substack{y^i,y^j \in \C_\nu\\{\nu=1,\dots,\mu}}} \left(\frac{\vx^{y^j}}{\vx^{y^i}}\right)^{\alpha_{ij}}
                          = \prod_{\substack{y^i,y^j \in \C_\nu\\{\nu=1,\dots,\mu}}} \beta_{ij}^{\alpha_{ij}},
  \end{equation*}
  and the system has a robust ratio in $y$ and $y'$. The second point follows immediately taking $y' = \mathbf{e}^i$ and $y = \mathbf{0}$.
\end{proof}

\noindent The objective of techniques for identifying ACR, including those of this paper, is to find as many robust complexes as possible and then establish whether $\mathbf{e}^i \in R$ for any $i=1,\dots,n$. Note that, in order to accomplish this, it is not necessary to find all robust ratios, which in general may be very difficult. Since we can define $R'$ as in \eqref{eq:robustness_space} using any subset of complexes $y$ and $y'$ known to have a robust ratio, and $R' \subseteq R$ necessarily holds, it suffices to show that $\mathbf{e}^i \in R'$.

We now restate the results on identification of complexes with robust ratio presented in \cite{ShinarFeinberg2010Robustness}.

\begin{theorem}
\label{thm:robustness_def0_def1}
Consider a chemical reaction network $\N$, with an associated mass action system $\M$.
\begin{enumerate}
\item
  If $\N$ has a deficiency of zero and is weakly reversible,
  then $\M$ has a robust ratio in each pair of complexes $y$ and $y'$ belonging to a common linkage class in $\N$.
\item
  If $\N$ has a deficiency of one, and $\M$ admits a positive steady state,
  then $\M$ has a robust ratio in every pair of nonterminal complexes $y$ and $y'$ in $\N$.
\end{enumerate}
\end{theorem}

The most general result presented in~\cite{ShinarFeinberg2010Robustness} about identification of robust ratios and species relies on the identification of a \emph{direct decomposition} of the network $\mathcal{N}$, i.e.\ a partition of the set of reactions $\mathcal{R}$ into subsets $\mathcal{R}_1, \dots, \mathcal{R}_k$ such that, if $\mathcal{N}_i$ is the network consisting of the reactions in $\mathcal{R}_i$, and $s_i$ is the rank of $\mathcal{N}_i$, then the $s_i$ sum to the total rank $s$ of $\mathcal{N}$. Since an equilibrium of the mass action system associated to $\mathcal{N}$ is an equilibrium for each sub-mass action system identified by the $\mathcal{N}_i$, then if one can prove that a complex ratio is robust in a subnetwork, then this ratio must be robust for the overall network.
The application of the deficiency zero and deficiency one criteria for the identification of robust ratios to the subnetworks $\mathcal{N}_i$ can therefore reveal robustness for the larger network.

Reconsider our earlier examples.

\begin{example}
\label{ex:example4}
  Reconsider the network from Examples~\ref{ex:example1} and~\ref{ex:example3}.
  The network is deficiency one and the nonterminal complexes are $A+B$ and $B$. Consequently, these complexes have a robust ratio by Theorem~\ref{thm:robustness_def0_def1}, part $2$.
  Since $A+B$ and $B$ differ in the species $A$, ACR in species $A$ follows.
\end{example}

\begin{example}
\label{ex:envz2}
Reconsider the network considered in Example \ref{ex:envz}.
  The network has deficiency $2$, and admits a subnetwork of deficiency zero given by the two reactions $\AI \xrightleftharpoons[r_2]{r_1} \AII$.
  By Theorem~\ref{thm:robustness_def0_def1}, part $1$, applied to this subnetwork, we conclude that the complexes $\AI$ and $\AII$ have a robust ratio.
  We therefore have that $R' = \mbox{span} \{ \AI - \AII \} \subseteq R$.
  Note that $R'$ is based on the set of known robust ratios, whereas the true set of robust ratios may produce a larger set $R$. It is clear that there is no species $X_i \in \mathcal{S}$ such that $\mathbf{e}^i \in R'$. ACR can therefore not be guaranteed by Theorem \ref{thm:robustness_def0_def1}.

\begin{sloppypar}
  Nevertheless, the system permits ACR in the species $Y_p$. This result was obtained in \citet{ShinarFeinberg2010Robustness} by calculating the concentration of $Y_p$ at steady states directly from the steady state equations:
\end{sloppypar}
  \begin{equation}
  \label{eq:yp_value}
  x_{\AVII} = \frac{k_1k_3k_5(k_{10}+k_{11})(k_{13}+k_{14})}{k_1k_3k_9k_{11}(k_{13}+k_{14})+k_2(k_4+k_5)(k_{10}+k_{11})k_{12}k_{14}}.
  \end{equation}
 An alternative derivation for~\eqref{eq:yp_value} based on the calculation of linear complex invariants is presented in~\cite{karp2012complex}.
  In the following sections, we will develop a method which allows us to expand the set of complexes having a robust ratio and use these to establish ACR in species $\AVII$. The developed method will also allow us to compute the ACR value \eqref{eq:dynamics} directly from a graphical representation known as a generalized chemical reaction network.
\end{example}

\section{Main Results}\label{sec:translated}

In this section, we will develop a new method by which to establish robust ratios of complexes and ACR of species. This method is based upon the method of \emph{network translation} introduced by \cite{johnston2014translated} which utilizes the theory of \emph{generalized chemical reaction networks} introduced by \cite{muller2012generalized}. We introduce these concepts now.

\subsection{Generalized chemical reaction networks}

\begin{sloppypar}
The following generalization to the notion of a chemical reaction network was introduced in \citet{muller2012generalized}.
\end{sloppypar}

\begin{definition}
\label{def:gcrn}
    A \emph{generalized chemical reaction network} $\N = (\S, \C, \C_K, \R)$ is a chemical reaction network $\N = (\S, \C, \R)$, with source complex set $\C\R$, together with a set of complexes $\C_K$, called \emph{kinetic complexes}, and a bijective map $h: \C\R \rightarrow \C_K$, called the \emph{kinetic map}.
\end{definition}

\noindent All of the connectivity concepts introduced in Section \ref{sec:crn} can be extended to generalized mass action networks.
   Some concepts, however, depend upon whether the stoichiometric complex set $\C$ or kinetic complex set $\C_K$ is chosen. 
We define the \emph{stoichiometric subspace} $S$ and \emph{stoichiometric deficiency} $\delta$ to be associated to the network
$(\S, \C,\R)$. For weakly reversible networks, we define the \emph{kinetic-order subspace} $S_K$ and the \emph{kinetic deficiency} $\delta_K$ to be associated to the network $(\S, \C_K, \R)$.

A \emph{generalized mass action system} $\M = (\S, \C, \C_K, \R, k)$ on $\N$ is constructed by using the complexes $\C$ to determine the stoichiometry, and the kinetic complexes $\C_K$ identified by the map $h$ to determine the reaction rates. The system of ODEs associated to $\M$ is given by
\begin{equation}\label{eq:dyn_Ak_gen}
  \frac{d\vx}{dt} = \mY \mA_\kappa \Psi_K(\vx),
\end{equation}
where $\mY$ is the complex matrix associated with $\C$ and $\Psi_K(\mathbf{x}) =(\vx^{h(y^1)}, \dots, \vx^{h(y^c)})$,
with the map $h$ being arbitrarily extended to $\C \setminus \C\R$.
The generalized mass action system in~\eqref{eq:dyn_Ak_gen} differs from~\eqref{eq:dyn_Ak} only in the monomial vector $\Psi(\vx)$.

We represent generalized mass action networks as graphs, with
each node enclosing both a stoichiometric complex and its associated kinetic complex, and with the latter indicated in parentheses as
in~\eqref{eq:proper_acr_intro_translated}. This allows for a compact
representation of both the graph $(\C, \R)$ and the graph $(\C_K, \R)$. Consider the following examples.

\begin{example}
\label{example6}
Consider the following generalized chemical reaction network:
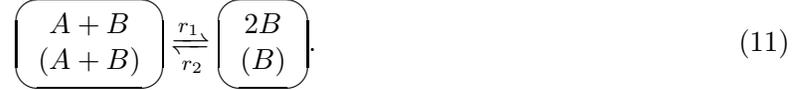
\begin{equation}
\label{eq:gen1}
    \mbox{\ovalbox{$\begin{array}{c} A+B \\ (A+B) \end{array}$}} \xrightleftharpoons[r_2]{r_1} \mbox{\ovalbox{$\begin{array}{c} 2B \\ (B) \end{array}$}}.
\end{equation}
We have the following sets: $\S = \{ A, B \}$, $\C = \{ A+B, 2B\}$, $\C_K = \{ A+B, B \}$, and $\R = \{ A+B \to 2B, 2B \to A+B\}$. The mapping $h$ is defined by reading the complexes in the boxes: $h(A+B) = A+B$ and $h(2B)=B$. We furthermore have $S = \mbox{span} \{ (1,-1) \}$, $S_K = \mbox{span} \{ (1,0) \}$, $\delta = \delta_K = 0$, and that the network is weakly reversible.

The corresponding generalized mass action system $\M = (\S, \C, \C_K, \R, k)$ is governed by the following dynamical equations \eqref{eq:dyn_Ak_gen}:
\begin{equation}
\label{eq:11}
\begin{split} \dot{x}_A & = -k_1 x_A x_B + k_2 x_B, \\ \dot{x}_B & = k_1 x_A x_B - k_2 x_B. \end{split}
\end{equation}
Notably, we have that \eqref{eq:11} coincides with \eqref{eq:n1}. That is, the systems are dynamically equivalent.
\end{example}

\subsection{Translated chemical reaction networks}

The realization that a mass action system can have the same system of ODEs as a generalized mass action system (see Example \ref{example6}) motivated the introduction of \emph{translated chemical reaction networks}~\citep{johnston2014translated,johnston2015computational}.
\begin{definition}
\label{def:translation}
  Let $\N = (\S, \C, \R)$ be a chemical reaction network with source complex set $\C\R$. Consider a generalized chemical reaction network $\tN = (\tS, \tC, \tC_{K},
  \tR)$ with source complex set $\widetilde{\C\R}$ and kinetic map $h: \widetilde{\C\R} \rightarrow \tC_{K} \subseteq \C\R$. Then $\tN$ is a \emph{translation} of the network $\N$ if there exist surjective maps $f: \R \rightarrow \tR$ and $g: \C\R \rightarrow \widetilde{\C\R}$ such that
    \begin{itemize}
        \item[(i)] if $f(r) = \tilde{r}$ for $r: y \rightarrow y'$ and $\tilde{r}: \tilde{y} \rightarrow \tilde{y}'$, then $y' - y = \tilde{y}' - \tilde{y}$;
        \item [(ii)] if $r: y \rightarrow y'$, $f(r)=\tilde{r}$, and $g(y) = \tilde{y}$, then $\tilde{r}: \tilde{y} \rightarrow \tilde{y}'$ for some $\tilde{y}' \in \tC$; and
        \item[(iii)] $g(h(\tilde{y})) = \tilde{y}$ for all $\tilde{y} \in \widetilde{\C\R}$.
    \end{itemize}
\end{definition}

The process of network translation can be visualized by adding complexes to both the source and product side of a reaction. Such a \emph{translation scheme} associates reactions to generalized reactions in the following way:
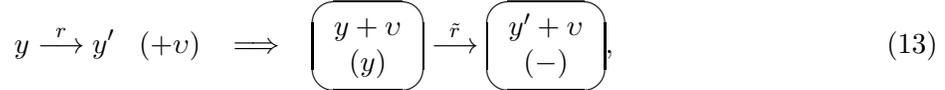
\begin{equation}
\label{eq:prop_trans}
y \stackrel{r}{\longrightarrow} y' \; \; \; (+ \upsilon) \; \; \; \Longrightarrow \; \; \; \mbox{\ovalbox{$\begin{array}{c} y+\upsilon \\ (y) \end{array}$}} \stackrel{\tilde{r}}{\longrightarrow} \mbox{\ovalbox{$\begin{array}{c} y'+\upsilon \\ (-) \end{array}$}},
\end{equation}
where the \emph{translation complex} is denoted $\upsilon \in \mathbb{Z}_{\geq 0}^n$.
That is, if we translate a reaction by $\upsilon$, the source complex $y$ becomes the kinetic complex associated to the stoichiometric complex $y+\upsilon$, which reacts to the product complex $y' + \upsilon$.

The translation scheme \eqref{eq:prop_trans} implicitly defines the mappings $f$ and $g$, according to Definition \ref{def:translation}. Specifically, we have $f(r) = \tilde{r}$ and $g(y) = y+\upsilon$. The mapping $h$, however, may not always be immediately defined.
For example, suppose that the source complexes of two reactions $r: y \to y'$ and $\hat{r}: \hat{y} \to \hat{y}'$ are translated to a common complex $\tilde{y}$. We represent this as follows:
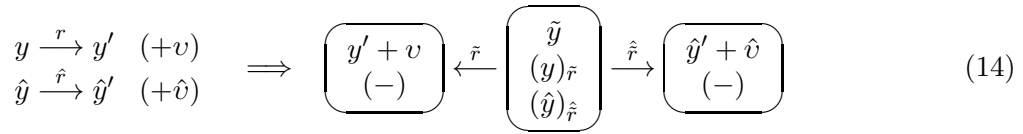
\begin{equation}
\label{eq:improp_trans}
\begin{array}{l} y \stackrel{r}{\longrightarrow} y' \; \; \; (+ \upsilon) \\ \hat{y} \stackrel{\hat{r}}{\longrightarrow} \hat{y}' \; \; \; (+ \hat{\upsilon}) \end{array} \; \; \; \Longrightarrow \; \; \; \mbox{\ovalbox{$\begin{array}{c} y'+ \upsilon\\ (-) \end{array}$}}\stackrel{\tilde{r}}{\longleftarrow} \mbox{\ovalbox{$\begin{array}{c} \tilde{y} \\ (y)_{\tilde{r}} \\ (\hat{y})_{\hat{\tilde{r}}} \end{array}$}} \stackrel{\hat{\tilde{r}}}{\longrightarrow} \mbox{\ovalbox{$\begin{array}{c} \hat{y}'+\hat{\upsilon} \\ (-) \end{array}$}}
\end{equation}
where $\tilde{y} = y + \upsilon = \hat{y} + \hat{\upsilon}$ and $\upsilon, \hat{\upsilon} \in \mathbb{Z}_{\geq 0}^n$ are the \emph{translation complexes} associated with reactions $r$ and $\hat{r}$, respectively. We have $f(r) = \tilde{r}$, $f(\hat{r}) = \hat{\tilde{r}}$, and $g(y) = g(\hat{y}) = \tilde{y}$.
If $y \neq \hat{y}$, the complex $\tilde{y}$ is therefore associated with two source complexes, $y$ and $\hat{y}$, and the respective reactions are indicated with subscripts in \eqref{eq:improp_trans}.

The translation scheme \eqref{eq:improp_trans} produces two candidates for the kinetic complex of $\tilde{y}$. We therefore require the following refinements of network translation.

\begin{definition}
\label{def:proper}
A translation $\tN$ of $\N$ is said to be \emph{proper} if $g$ is injective, and \emph{improper} otherwise. If $\tN$ is an improper translation of $\N$, we define the \emph{improper reaction set} $\mathcal{R}_I \subseteq \mathcal{R}$ to be given by
\[ \mathcal{R}_I = \left\{ y \to y' \in \mathcal{R} \; | \; y \not\in \tC_{K} \right\}.\]
\end{definition}

\noindent For proper translations, each kinetic complex in the resulting graph is uniquely defined by $g$ so that $h = g^{-1}$. In the scheme \eqref{eq:prop_trans}, this corresponds to $h(y+\upsilon)=y$.

If a translation is improper, as in \eqref{eq:improp_trans}, we are only able to define a generalized chemical reaction network after restricting to a single source complex in the preimage $g^{-1}$. For example, defining $h(\tilde{y}) = y$ in \eqref{eq:improp_trans} gives
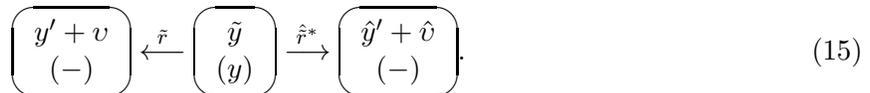
\begin{equation}
\label{eq:improp_trans_choice}
\mbox{\ovalbox{$\begin{array}{c} y' + \upsilon \\ (-) \end{array}$}}\stackrel{\tilde{r}}{\longleftarrow} \mbox{\ovalbox{$\begin{array}{c} \tilde{y} \\ (y) \end{array}$}} \stackrel{\hat{\tilde{r}}^*}{\longrightarrow} \mbox{\ovalbox{$\begin{array}{c} \hat{y}'+\hat{\upsilon} \\ (-) \end{array}$}}.
\end{equation}
In this case, we have that $\hat{r} \in \mathcal{R}_I$ since $\hat{y} \not\in \tC_K$, i.e. the source complex of $\hat{r}$ is not a kinetic complex of the corresponding reaction in the translation. We indicate this with an asterisk in the generalized graph of complexes. Note that the translation would be equally valid if we had instead chosen $h(\tilde{y}) = \hat{y}$, in which case $r \in \mathcal{R}_I$.

When discussing network translations, it will be convenient to adopt specialized notation for the rate constant set $\tilde{k} \in \mathbb{R}_{> 0}^{\tilde{m}}$ in the corresponding generalized mass action system $\tM = (\tS, \tC, \tC_K, \tR, \tilde{k})$. We introduce the following.

\begin{definition}
\label{def:trans_mas}
Consider a translation $\tN$ of a chemical reaction network $\N$ with corresponding mass action system $\M$ and associated rate constant vector $k \in \mathbb{R}_{> 0}^m$. We define the \emph{translated mass action system} to be the generalized mass action system $\tM = (\tS, \tC, \tC_K, \tR, \tilde{k})$ taken with the following rate constants:
\begin{equation}\label{eq:translation_kin_par}
  \tilde{k}_{\tilde{r}} = \sum_{\substack{r \in \R \setminus \mathcal{R}_I\\f(r) =\tilde{r}}}k_{r} + \sum_{\substack{r \in \mathcal{R}_I\\f(r) = \tilde{r}}}k_{r}^*.
\end{equation}
The translated chemical reaction system $\tM$ and mass action system $\M$ are furthermore said to be: (a) \emph{dynamically equivalent} if the dynamical equations \eqref{eq:dyn_Ak} and \eqref{eq:dyn_Ak_gen} coincide; and (b) \emph{steady state equivalent} if there is a choice of values for $k_r^*$ for $r \in \mathcal{R}_I$ such that the dynamical equations \eqref{eq:dyn_Ak} and \eqref{eq:dyn_Ak_gen} have the same positive steady states. If the translation is improper and case (b) holds, the translation is said to be \emph{resolvable}.
\end{definition}

For proper translations, the corresponding translated mass action system $\tM$ is necessarily dynamically equivalent to the mass action system $\M$ corresponding to the original network \citep[Lemma 2]{johnston2014translated}. For improper translations, the systems may not be dynamically equivalent; however, sufficient conditions for resolvability were presented in \cite{johnston2014translated,johnston2015computational}.

Consider the following examples.
\begin{example}
\label{ex:example1_transl}
Reconsider the network from Examples \ref{ex:example1}, \ref{ex:example3} and \ref{ex:example4}. The translation scheme
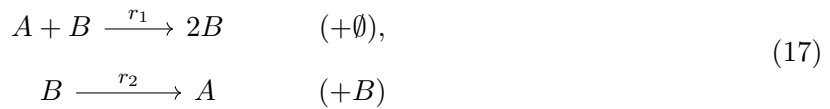
\begin{equation}
\label{eq:network1_scheme}
  \begin{tikzcd}
    A+B \arrow[r, "r_1"] & 2B & (+\emptyset), \\[-0.2in]
    B \arrow[r, "r_2"] & A & (+B)
  \end{tikzcd}
\end{equation}
directly generates the generalized reaction network \eqref{eq:gen1} considered in Example \ref{example6}.
Formally, we have that $f(r_1) = \tilde{r}_1$, $f(r_2) = \tilde{r}_2$, $g(A+B) = A+B$, and $g(B) = 2B$. Since $g$ is injective, the translation is proper and the kinetic mapping $h$ is defined as $h = g^{-1}$. As previously noted, $\M$ and $\tM$ are dynamically equivalent.

\end{example}

\begin{example}
\label{ex:envz3}
Reconsider the network from Examples \ref{ex:envz} and \ref{ex:envz2}. In Section $5.4$ of~\cite{johnston2014translated}, the following translation scheme was proposed:
      
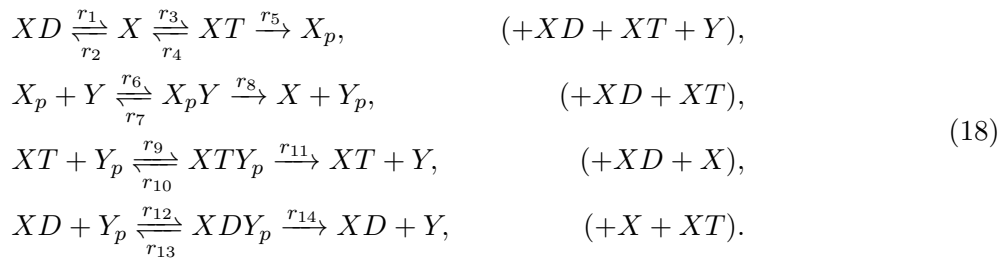
\begin{equation}\label{eq:envz_ompr_trans2}
        \begin{aligned}
          & \AI \xrightleftharpoons[r_2]{r_1} \AII \xrightleftharpoons[r_4]{r_3} \AIII \xrightarrow{r_5} \AIV,  \; \; \; & ( + \AI + \AIII + \AV),\\
          & \AIV + \AV \xrightleftharpoons[r_7]{r_6} \AVI \xrightarrow{r_8} \AII + \AVII, \; \; \; & ( + \AI + \AIII),\\
          & \AIII + \AVII \xrightleftharpoons[r_{10}]{r_9} \AVIII \xrightarrow{r_{11}} \AIII + \AV, \; \; \; & ( + \AI + \AII ), \\
          & \AI + \AVII \xrightleftharpoons[r_{13}]{r_{12}} \AIX \xrightarrow{r_{14}} \AI + \AV, \; \; \; & ( + \AII + \AIII).
        \end{aligned}
    \end{equation}
The translation scheme \eqref{eq:envz_ompr_trans2} yields the following:
      \begin{equation}\label{eq:envz_ompr_trans1}
        \begin{tikzcd}[column sep=0.5cm]
            \mbox{\ovalbox{$\begin{array}{c} 2\AI + \AIII + \AV \\ (\AI) \end{array}$}} \arrow[r, rightharpoonup, yshift=+0.2ex, "\tilde{r}_{1}"]
          & \mbox{\ovalbox{$\begin{array}{c}\AI + \AII + \AIII + \AV \\ (\AII) \end{array}$}} \arrow[l, rightharpoonup, yshift=-0.2ex, "\tilde{r}_{2}"] \arrow[r, rightharpoonup, yshift=+0.2ex, "\tilde{r}_{3}"]
          & \mbox{\ovalbox{$\begin{array}{c} \AI + 2\AIII + \AV \\ (\AIII) \end{array}$}} \arrow[l, rightharpoonup, yshift=-0.2ex, "\tilde{r}_{4}"] \arrow[d, "\tilde{r}_{5}"] & \\
            \mbox{\ovalbox{$\begin{array}{c} \AII + \AIII + \AIX \\ (\AIX) \end{array}$}} \arrow[ur, "\tilde{r}_{14}"] \arrow[dr, rightharpoonup, "\tilde{r}_{13}", xshift=+0.6ex]
          & \mbox{\ovalbox{$\begin{array}{c} \AI + \AII + \AVIII \\ (\AVIII) \end{array}$}} \arrow[u, "\tilde{r}_{11}"] \arrow[d, rightharpoonup, "\tilde{r}_{10}", xshift=+0.2ex]
          & \mbox{\ovalbox{$\begin{array}{c} \AI + \AIII + \AIV + \AV \\ (\AIV + \AV) \end{array}$}} \arrow[d, rightharpoonup, "\tilde{r}_{6}", xshift=+0.2ex] \\
          &  \mbox{\ovalbox{$\begin{array}{c} \AI + \AII + \AIII + \AVII \\ (\AIII + \AVII)_{\tilde{r}_9} \\ (\AI + \AVII)_{\tilde{r}_{12}} \end{array}$}} \arrow[ul, rightharpoonup, "\tilde{r}_{12}", xshift=-0.6ex] \arrow[u, rightharpoonup, "\tilde{r}_{9}", xshift=-0.2ex]
          & \mbox{\ovalbox{$\begin{array}{c} \AI + \AIII + \AVI \\ (\AVI) \end{array}$}} \arrow[l, "\tilde{r}_8"'] \arrow[u, rightharpoonup, "\tilde{r}_{7}", xshift=-0.2ex] &
      \end{tikzcd}
    \end{equation}
\noindent Since $\AIII + \AVII$ and $\AI + \AVII$ are both translated to $\AI + \AII + \AIII + \AVII$, $g$ is not injective and the translation is improper.
We may select either of the translated source complexes to be the corresponding kinetic complex.
If we define $h(\AI + \AII + \AIII + \AVII) = \AIII + \AVII$ and $h(\tilde{y}) = g^{-1}(\tilde{y})$ for all other $\tilde{y} \in \tC$ we arrive at the following generalized chemical reaction network:
      
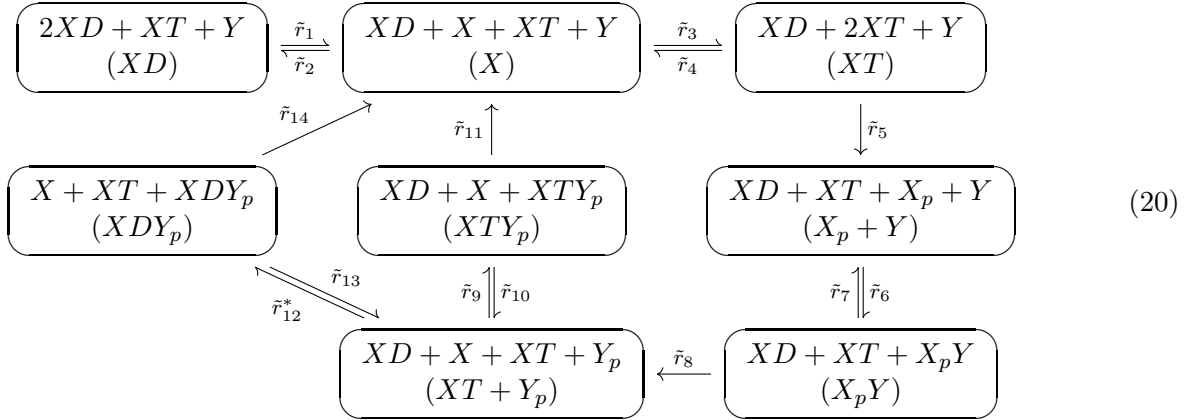
\begin{equation}\label{eq:envz_ompr_gma}
        \begin{tikzcd}[column sep=0.5cm]
            \mbox{\ovalbox{$\begin{array}{c} 2\AI + \AIII + \AV \\ (\AI) \end{array}$}} \arrow[r, rightharpoonup, yshift=+0.2ex, "\tilde{r}_{1}"]
          & \mbox{\ovalbox{$\begin{array}{c}\AI + \AII + \AIII + \AV \\ (\AII) \end{array}$}} \arrow[l, rightharpoonup, yshift=-0.2ex, "\tilde{r}_{2}"] \arrow[r, rightharpoonup, yshift=+0.2ex, "\tilde{r}_{3}"]
          & \mbox{\ovalbox{$\begin{array}{c} \AI + 2\AIII + \AV \\ (\AIII) \end{array}$}} \arrow[l, rightharpoonup, yshift=-0.2ex, "\tilde{r}_{4}"] \arrow[d, "\tilde{r}_{5}"] & \\
            \mbox{\ovalbox{$\begin{array}{c} \AII + \AIII + \AIX \\ (\AIX) \end{array}$}} \arrow[ur, "\tilde{r}_{14}"] \arrow[dr, rightharpoonup, "\tilde{r}_{13}", xshift=+0.6ex]
          & \mbox{\ovalbox{$\begin{array}{c} \AI + \AII + \AVIII \\ (\AVIII) \end{array}$}} \arrow[u, "\tilde{r}_{11}"] \arrow[d, rightharpoonup, "\tilde{r}_{10}", xshift=+0.2ex]
          & \mbox{\ovalbox{$\begin{array}{c} \AI + \AIII + \AIV + \AV \\ (\AIV + \AV) \end{array}$}} \arrow[d, rightharpoonup, "\tilde{r}_{6}", xshift=+0.2ex] \\
          &  \mbox{\ovalbox{$\begin{array}{c} \AI + \AII + \AIII + \AVII \\ (\AIII + \AVII) \end{array}$}} \arrow[ul, rightharpoonup, "\tilde{r}_{12}^*", xshift=-0.6ex] \arrow[u, rightharpoonup, "\tilde{r}_{9}", xshift=-0.2ex]
          & \mbox{\ovalbox{$\begin{array}{c} \AI + \AIII + \AVI \\ (\AVI) \end{array}$}} \arrow[l, "\tilde{r}_8"'] \arrow[u, rightharpoonup, "\tilde{r}_{7}", xshift=-0.2ex] &
      \end{tikzcd}
    \end{equation}
We have $\tilde{\delta}=\tilde{\delta}_K = 0$ and that the network is weakly reversible. Note that the source complex $XD + Y_p$ of $r_{12}$ is not the kinetic complex of the corresponding reaction in \eqref{eq:envz_ompr_gma}, so that $\mathcal{R}_I = \{ r_{12}\}$. This is indicated with an asterisk.
Although the $\M$ and $\tM$ are not dynamically equivalent, we will be able to show that the translation is resolvable. Consequently, we will be able to establish steady state properties of $\M$ by analyzing the generalized system \eqref{eq:envz_ompr_gma}.

\end{example}

\subsection{Translations and robustness}\label{sec:transl_rob}

In this section we establish relationships between robustness and network translation. We start with the following result, which is proved in Appendix~\ref{app:proofs}.
\begin{theorem}
\label{thm:robustness_translation}
Consider a chemical reaction network $\N$, with an associated mass action system $\M$,
and suppose that $\M$ admits a positive steady state.
Let $\tN$ be a translation of $\N$ which is either proper, or improper and resolvable. Then:
\begin{enumerate}
\item
  If $\tN$ has a stoichiometric deficiency of zero and is weakly reversible,
  then $\M$ has a robust ratio
  in each pair of complexes $y$ and $y'$ which are translated to a common linkage class in $\tN$.
\item
  If $\tN$ has a stoichiometric deficiency of one,
  then $\M$ has a robust ratio in every pair of complexes $y$ and $y'$ which are translated to nonterminal complexes in $\tN$.
\end{enumerate}
\end{theorem}

\noindent ACR of a species can be established by applying part $1$ or $2$ of Theorem~\ref{thm:robustness_translation} to determine complexes with robust ratios and then applying Lemma \ref{lemma1}.

For weakly reversible deficiency zero networks, we may use the network structure to not only establish that robustness holds, according to condition $1$ of Theorem \ref{thm:robustness_translation}, but also (a) to calculate the robust value at steady state, and (b) as a method for determining resolvability of translations. To do so, we introduce \emph{tree constants} $K(y)$, which are polynomials in the rate constants $k=(k_1, \ldots, k_m)$ of a weakly reversible network \citep{craciun2009toric,johnston2014translated}.
To calculate the tree constant $K(y)$ associated to the complex $y$, we let $\mathcal{L}(y)$ denote the linkage class containing $y$
and define a \emph{spanning ${y}$-tree} $\mathcal{T}$ to be a subgraph of the graph of complexes that contains all complexes in the linkage class, admits no cycles, and has the complex $y$ as its unique sink.
Denote by $\mathcal{T}(y)$ the set of all spanning ${y}$-trees.
Then $K(y)$ is given by
\begin{equation}\label{def:tc}
  K(y) = \sum_{\mathcal{T} \in \mathcal{T}(y)} \prod_{r \in \mathcal{T}} k_r.
\end{equation}
Equivalently, $K(y)$ can be found, modulo a sign, by considering the
restriction of the kinetic matrix $\mA_\kappa$ to $\mathcal{L}(y)$, and
calculating the minor obtained by deleting the column corresponding to $y$ and
any of the rows~\citep{craciun2009toric,johnston2014translated}.

For translated networks $\tN$ and the corresponding translated mass action systems $\tM$, we define $\tilde{K}(\tilde{y})$ equivalently to \eqref{def:tc} where $\tilde{y}$ is the stoichiometric complex occupying a given vertex and the rate constants $\tilde{k}$ are defined by \eqref{eq:translation_kin_par}.

We have the following result regarding the computation of the steady state value of
robust ratios. The proof is given in Appendix~\ref{app:proofs2}.

\begin{theorem}
\label{thm:resolvability_translation}
Consider a chemical reaction network $\N$ with corresponding mass action system $\M$ and rate constant vector $k \in \mathbb{R}_{> 0}^m$. Let $\tN$ be a translation of $\N$ which has a stoichiometric deficiency of zero and is weakly reversible, and let $\tM$ be an associated translated mass action system with rate constant vector $\tilde{k} \in \mathbb{R}_{> 0}^{\tilde{m}}$ defined according to \eqref{eq:translation_kin_par}. Then we have the following:
\begin{enumerate}
\item
  Suppose that $h(\tilde{y}) = y$ and $h(\tilde{y}')= y'$, where $\tilde{y}, \tilde{y}' \in \tC$ belong to a common linkage class in $\tN$.
  Also suppose one of the following:
\begin{enumerate}
\item[(a)]
$\tN$ is proper; or
\item[(b)] 
$\tN$ is improper and resolvable with the rate constants $k_r^*$, $r \in \mathcal{R}_I$, in \eqref{eq:translation_kin_par} chosen so that steady state equivalence of $\M$ and $\tM$ holds; or
\item[(c)]
$\tN$ is improper and the ratio of
tree constants $\tilde{K}(\tilde{y})/\tilde{K}(\tilde{y}')$ does not depend upon any rate constant $k_r^*$ corresponding to a reaction $r \in \mathcal{R}_I$.
\end{enumerate}
Then $y$ and $y'$ have a robust ratio in $\M$ and, in particular, at every positive steady state of $\M$, we have
\[\frac{\mathbf{x}^{y}}{\mathbf{x}^{y'}} = \frac{\tilde{K}(\tilde{y})}{\tilde{K}(\tilde{y}')}.\]
\item
\begin{enumerate}
\item[(a)]
Suppose that $\tN$ is improper and that there is a pair of complexes $y, y' \in \C$ such that $g(y) = g(y')$ for which
\[y - y' = \sum_{\theta=1}^{\tilde{l}} \mathop{\sum_{i,j=1}^{\tilde{c}}}_{\tilde{y}^i, \tilde{y}^j \in \tilde{\mathcal{L}}_\theta} c_{ij} (h(\tilde{y}^{i}) - h(\tilde{y}^{j}))\]
where $\tilde{\mathcal{L}}_\theta$, $\theta=1, \ldots, \tilde{l}$, are the linkage classes of $\tN$ and $c_{ij} \in \mathbb{R}$ are constants. Suppose also that the \emph{kinetic adjustment factor} $\tilde{K}(y,y')$ defined by
\begin{equation}
\label{def:kin_adj_factor}
\tilde{K}(y,y') = \prod_{\theta=1}^{\tilde{l}} \mathop{\prod_{i,j=1}^{\tilde{c}}}_{y_i, y_j \in \mathcal{L}_\theta} \left( \frac{\tilde{K}(\tilde{y}^i)}{\tilde{K}(\tilde{y}^j)} \right)^{c_{ij}}
\end{equation}
does not depend on any rate constant $k_r^*$ corresponding to a reaction $r \in \mathcal{R}_I$. Then $y$ and $y'$ have a robust ratio in $\M$ and, in particular, we have
\[\frac{\mathbf{x}^{y}}{\mathbf{x}^{y'}} = \tilde{K}(y,y').\]
\item[(b)]
Suppose that, for all reactions $r = y \to y' \in \mathcal{R}_I$, the conditions of part 2(a) hold for the source complex $y$ and the corresponding kinetic complex $h(g(y))$. Then $\tN$ is resolvable. In particular, the mass action system $\M$ is steady state equivalent to the translated mass action system $\tM$ with $k_i^* = \tilde{K}(y,h(g(y)))k_i$ where $\tilde{K}(y,h(g(y)))$ is the kinetic adjustment factor \eqref{def:kin_adj_factor}.
\end{enumerate}
\end{enumerate}
\end{theorem}

\noindent For improper translations, Theorem \ref{thm:resolvability_translation} can be used to generate robust ratios directly (part 1(c) and 2(a)), or to first demonstrate resolvability through part 2(b) and then establish robust ratios through part 1(b). Notably, unlike previous results stemming from network translation, Theorem \ref{thm:resolvability_translation} part 1(c) does not depend upon the translation being resolvable to guarantee that certain complexes have robust ratios.

\section{Computational approach}
\label{sec:translation_algo}

\noindent When applying Theorems \ref{thm:robustness_translation} and \ref{thm:resolvability_translation}, the translation $\tN$ is typically not a priori known and therefore must be constructed. In this section we describe an algorithmic approach to such a construction.
The ideas outlined here can be used to encode the problem as a mixed integer linear program (MILP).
A more detailed description of the framework can be found in Appendix~\ref{app:code}.

 The method for identification of translations that we are going to describe relies on the computation of the \emph{elementary modes} of the network.
In order to introduce this definition, we first need to define a pair of structural matrices relevant to a chemical reaction network $\N$. For fixed reaction and complex orders $\mathcal{R} = \{ r_1, \ldots, r_m \}$ and $\mathcal{C} = \{ y^1, \ldots, y^c\}$, we define: (1) the {\it stoichiometric matrix} $\mGamma \in \mathbb{Z}^{n \times m}$ to be the matrix where the $i^{th}$ column is given by $\Gamma_{i} = y'-y$ for $r_i : y \to y'$; and (2) the \emph{incidence matrix} $\mI_a \in \mathbb{Z}^{c \times m}$ to be the matrix with entries ${(\mI_a)}_{ij} = -1$ if $y^i$ is the reactant complex of reaction $r_j$, ${(\mI_a)}_{ij} = 1$ if $y^i$ is the product of reaction $r_j$, and ${(\mI_a)}_{ij} = 0$ otherwise. It can be shown that $\mGamma = \mY \mI_a$ where $\mY$ is the complex matrix of $\N$. The \emph{support} of a vector $w \in \mathbb{R}^m$ is defined as the set $\{i=1,\dots,m\ | \ w_i\neq 0\}$.

The set $\{v \in \mathbb{R}^{m}_{\geq 0}\ |\ v \in \mbox{ker}(\mGamma)\}$ is a pointed polyhedral cone, and admits a finite set of generators which are unique up to scalar multiplication.
We call the generators \emph{elementary modes} for $\N$; they are the vectors of the cone with minimal support with respect to inclusion.
An elementary mode $w$ and is said to be \emph{cyclic} if $w \in \mbox{ker}(\mI_a)$, and \emph{stoichiometric} otherwise.

\subsection{Linear Programming Framework}

\noindent The goal of the algorithm is to identify translation complexes $\upsilon^1, \ldots, \upsilon^m \in \mathbb{Z}^m$, which translate the reactions $r_1,
\ldots, r_m$ according to the translation scheme \eqref{eq:prop_trans}, to construct a translated network which has as low of a stoichiometric deficiency as possible.
The observation behind the method, made in~\citet[Remark $11$]{johnston2014translated}, can be restated as follows.
If the deficiency of a network is zero, then every elementary mode is a cycle.
If instead the network has strictly positive deficiency, then it may possess a stoichiometric elementary mode.
To find a network with lower deficiency, we therefore attempt to convert stoichiometric elementary modes in the original network into cyclic elementary modes in the translation.

The method applies to reaction networks admitting stoichiometric elementary modes that satisfy the hypotheses of the following result.

\begin{lemma}\label{lem:ems_to_cycles}
  Let $\N$ be a chemical reaction network.
  Suppose that $w \in \mbox{ker}(\mGamma) \cap \mathbb{R}^m_{\geq 0}$ has support on $I \subseteq \{ 1, \ldots, m \}$, and that $w_i = 1$ for all $i \in I$.
  Then there exists a translation $\tN$ for the network $\N$ such that $\tN$ admits a cycle with support on $I$.
\end{lemma}

\begin{proof}
Without loss of generality, assume that $w$ has support on the first $k = |I|$ reactions.
Since $w \in \mbox{ker}(\mGamma)$, we have that
\begin{equation}\label{eq:sum_r_sum_p}
  \mathop{\sum_{i = 1}^k}_{r_i: y \to y'} y = \mathop{\sum_{i = 1}^k}_{r_i: y \to y'} y'.
\end{equation}
Consider, for $r_1,\dots,r_k$, the translation complexes $\upsilon^1,\dots,\upsilon^k$ defined as
\[\upsilon^{j} = \mathop{\sum_{i=1}^{j-1}}_{r_i: y \to y'} y' + \mathop{\sum_{i=j+1}^k}_{r_i:y \to y'}y, \; j =1, \ldots, k.\]
This translation consists of all the products of the reactions preceding $r_j$,
  plus all the reactants of the reactions following $r_j$.
Call $\tilde{r}_1,\dots,\tilde{r}_k$ the resulting translated reactions.
  With this translation scheme, for the reactions $\tilde{r}_j = \tilde{y} \to \tilde{y}'$ and $\tilde{r}_{j+1} = \hat{y} \to \hat{y}'$, we have
    \[\tilde{y}' = \mathop{\sum_{i=1}^{j}}_{r_i: y \to y'} y' + \mathop{\sum_{i=j+1}^k}_{r_i: y \to y'} y = \hat{y}, \; j=1,\ldots, k-1\]
and for $\tilde{r}_k = \tilde{y} \to \tilde{y}'$ and $\tilde{r}_1 = \hat{y} \to \hat{y}'$ we have
\begin{equation}
\label{eq:second}
\tilde{y}' = \mathop{\sum_{i = 1}^k}_{r_i: y \to y'} y' = \mathop{\sum_{i = 1}^k}_{r_i: y \to y'} y = \hat{y}
\end{equation}
  where \eqref{eq:second} follows from~\eqref{eq:sum_r_sum_p}.
  It follows that the reactions $\tilde{r}_{1}, \dots, \tilde{r}_{k}$ define a cycle in $\tN$ and we are done.
\end{proof}

\begin{remark}\label{rmk:min_transl}
  Consider the setting of Lemma \ref{lem:ems_to_cycles}, and define the translation complex $\tilde{\upsilon}$ with entries
  \[\tilde{\upsilon}_l = \min_{1 \leq j \leq k}\left(\mathop{\sum_{i=1}^{j-1}}_{r_i: y \to y'} y_l' + \mathop{\sum_{i=j+1}^k}_{r_i: y \to y'} y_l\right).\]
  If the reactions $\tilde{r}_i$, $i=1,\dots,k$, are further  translated by the complex $-\tilde{y}$,
the result is still a cycle with complexes with nonnegative stoichiometry. {\hfill $\square$}
\end{remark}

\begin{remark}\label{rmk:order_reacts}
The translation scheme identified in Lemma~\ref{lem:ems_to_cycles} depends on the order chosen for the reactions in the support of the elementary mode $w$. {\hfill $\square$}
\end{remark}

Proposition~\ref{lem:ems_to_cycles} suggests a technique for identifying the translation complexes $\upsilon^1$, \ldots, $\upsilon^m$
as solutions of a system of linear equations. Specifically, suppose that the indices $\{j_1, \dots, j_k\} \subseteq \{ 1, \ldots, m\}$ identify a stoichiometric
elementary mode, and that $\hat{\upsilon}^{j_1}, \dots, \hat{\upsilon}^{j_k}$ are translation complexes specific to this stoichiometric elementary mode such that the
reactions $r_{j_1} + \hat{\upsilon}^{j_1}, \dots, r_{j_k} + \hat{\upsilon}^{j_k}$ form a cycle by Proposition \ref{lem:ems_to_cycles}.
Then, to guarantee that $r_{j_1}, \dots, r_{j_k}$ are translated to this cycle in $\tN$, we impose that
\begin{equation}\label{eq:translation_complexes}
  \begin{aligned}
	& \upsilon_i^{j_1} - \hat{\upsilon}_i^{j_1} = \upsilon_i^{j_2} - \hat{\upsilon}_i^{j_2} = \cdots = \upsilon_i^{j_k} - \hat{\upsilon}_i^{j_k}, \mbox{ for } i = 1, \dots, n.
  \end{aligned}
\end{equation}
The constraint set \eqref{eq:translation_complexes} only imposes that reactions in the stoichiometric elementary mode, $r_{j_1}, \ldots, r_{j_k}$, are translated to a cycle.
In general, we will have multiple stoichiometric elementary modes, each with their set of translated complexes $\hat{\upsilon}_{j_1}, \ldots, \hat{\upsilon}_{j_k}$, and it might not be possible
to satisfy the constraint \eqref{eq:translation_complexes} simultaneously for all such elementary stoichiometric modes.

The following are other considerations which are instrumental in generating the MILP system for determining a translation $\tN$ suitable for application of Theorems \ref{thm:robustness_translation} and \ref{thm:resolvability_translation} from a given chemical reaction network $\N$. The list of constraints capable of implementing these conditions in a MILP framework are presented in Appendix \ref{app:code}. The corresponding constraint sets are indicated in brackets.
\begin{itemize}
  \item[(i)] Reactions with a common source complex in $\N$ are translated by the same translation complex \textbf{(\hyperlink{c:same_reactant}{React})}.
  \item[(ii)] Reactions in a cycle in $\N$ are translated by the same translation complex~\textbf{(\hyperlink{c:cycle}{Cycle})}.
  \item[(iii)] Optionally, permutations of the reactions in the support of a stoichiometric mode can be considered \textbf{(\hyperlink{c:perm1}{Perm1})}, \textbf{(\hyperlink{c:perm2}{Perm2})}, \textbf{(\hyperlink{c:perm3}{Perm3})}, \textbf{(\hyperlink{c:perm4}{Perm4})}.
  \item[(iv)] Optionally, we can require the translation to be proper \textbf{(\hyperlink{c:proper1}{Proper1})}, \textbf{(\hyperlink{c:proper2}{Proper2})}.
  \item[(v)] The number of stoichiometric elementary modes that are translated to cycles is maximized \textbf{(\hyperlink{obj:max_transl}{Obj})}.
\end{itemize}

\subsection{Examples}

In this section, we illustrate the application, subtleties, and limitations of Theorem \ref{thm:robustness_translation}, Theorem \ref{thm:resolvability_translation}, and the computational algorithm outlined in Section \ref{sec:translation_algo} through examples.

\begin{example}
Consider the network \eqref{eq:proper_acr_intro} from the introduction. The translation~\eqref{eq:proper_acr_intro_translated} is proper, has a stoichiometric deficiency of zero and is weakly reversible.
Since the translated reaction network consists of a single linkage class, Theorem~\ref{thm:robustness_translation} part $1(a)$ guarantees that the complexes $A + B$, $C$, $D$ and $A$ have a robust ratio in $\M$. 
This implies the ACR in $B$. We can find the steady state value for $B$ by calculating the ratio between the tree constants $\tilde{K}(A+B)$ and $\tilde{K}(2A)$ in $\tN$. We have
\[\begin{split} \tilde{K}(A+B) & = k_3k_4k_5, \\ \tilde{K}(2A) & = (k_1+k_2)k_3k_4,\end{split}\]
so that
\[x_B = \frac{x_Ax_B}{x_A} = \frac{\tilde{K}(A+B)}{\tilde{K}(2A)} = \frac{k_5}{k_1+k_2}.\]
\end{example}

\begin{example}\label{ex:envz4}
  Reconsider the network from Examples \ref{ex:envz}, \ref{ex:envz2}, and \ref{ex:envz3}.
We previously used the translation scheme \eqref{eq:envz_ompr} to obtain the improper translated chemical reaction network \eqref{eq:envz_ompr_gma}.
We now use Theorem \ref{thm:resolvability_translation} to general robust ratios, establish resolvability, and compute the ACR value of $Y_p$ \eqref{eq:yp_value}.

First of all, notice that the kinetic complexes $\AIII+\AVII$ and $\AIII$ differ in only $\AIII$. We would be able to establish ACR in $\AVII$ directly by Theorem \ref{thm:resolvability_translation} part 1(c) if the ratio of tree constants for the corresponding complexes in \eqref{eq:envz_ompr_gma} did not depend on $k_{12}^*$ (since $\mathcal{R}_I = \{ r_{12} \}$). We have
\[\begin{split}
\tilde{K}(\AI + 2\AIII + \AV) & = k_1k_3k_6k_8(k_9k_{11}(k_{13}+k_{14})+(k_{10}+k_{11})k_{12}^*k_{14}),\\
\tilde{K}(\AI + \AII + \AIII + \AVII) & = k_1k_3k_5k_6k_8(k_{10}+k_{11})(k_{13}+k_{14}),
\end{split}\]
so that
\begin{equation}
\label{eq:int}
\frac{\tilde{K}(\AI + \AII + \AIII + \AVII)}{\tilde{K}(\AI + 2\AIII + \AV)} = \frac{k_5(k_{10}+k_{11})(k_{13}+k_{14})}{k_9k_{11}(k_{13}+k_{14})+(k_{10}+k_{11})k_{12}^*k_{14}}.
\end{equation}
Since this depends on $k_{12}^*$, we may not establish ACR directly from Theorem \ref{thm:resolvability_translation} part 1(c). We must first establish that the translation is resolvable by Theorem \ref{thm:resolvability_translation} part 2(b).

We determine complexes which have a robust ratio by Theorem \ref{thm:resolvability_translation} by computing ratios of tree constants from \eqref{eq:envz_ompr_trans1}. We have:
\[ \left\{
\begin{array}{ll}
\displaystyle{ \frac{\tilde{K}(2\AI+\AIII+\AV)}{\tilde{K}(\AI+\AII+\AIII+\AV)} = \frac{k_2}{k_1}, \; \; \;} & \displaystyle{\frac{\tilde{K}(\AI+\AII+\AIII+\AV)}{\tilde{K}(\AI+2\AIII+\AV)} = \frac{k_4+k_5}{k_3},} \\
\displaystyle{ \frac{\tilde{K}(\AI+2\AIII+\AV)}{\tilde{K}(\AI+\AIII+\AIV+\AV)} = \frac{k_6k_8}{k_5(k_7+k_8)}, \; \; \;} & \displaystyle{\frac{\tilde{K}(\AI+\AIII+\AIV+\AV)}{\tilde{K}(\AI+\AIII+\AVI)} = \frac{k_7+k_8}{k_6}.}
\end{array}
\right. \]
    Since these ratios do not depend on $k^*_{12}$, we have that the corresponding pairs of kinetic complexes have robust ratios. From part 1 of Theorem
\ref{thm:resolvability_translation}, we therefore have that $\M$ has a robust ratio in each pair of complexes in $\{\AI, \AII, \AIII, \AIV+\AV, \AVI \}$.

Since $g(\AIII+\AVII) = g(\AI+\AVII)$, we now consider Theorem \ref{thm:resolvability_translation} part 2(a).
Since the difference $(\AIII+\AVII) - (\AI+\AVII)=\AIII-\AI$ is in $R$, we have that $\AIII+\AVII$ and $\AI+\AVII$ have a robust ratio.
Specifically, we have
\begin{equation}
\label{envz_robust}
\frac{x_{\AI}x_{\AVII}}{x_{\AIII}x_{\AVII}} = \frac{\tilde{K}(2\AI+\AIII+\AV)}{\tilde{K}(\AI+2\AIII+\AV)} = \frac{k_2(k_4+k_5)}{k_1k_3}
\end{equation}
at every positive steady state of the mass action system associated with \eqref{eq:envz_ompr}. Since these were the only complexes mapped to the same complex via $g$, from Theorem \ref{thm:resolvability_translation} part 2(b) we have that the translation is resolvable.
In particular, the generalized mass action system associated with \eqref{eq:envz_ompr_trans1} is steady state equivalent to \eqref{eq:envz_ompr} for the rate constant choice $k_{12}^* = \frac{k_2(k_4+k_5)}{k_1k_3} k_{12}$. 
    We can then apply Theorem \ref{thm:resolvability_translation} part 1(b) to determine the value of $Y_p$ at steady state by
    substituting this value into \eqref{eq:envz_ompr_gma}. We have that
    \begin{equation*}
      \begin{aligned}
    x_{\AVII} & = \frac{k_5(k_{10}+k_{11})(k_{13}+k_{14})}{k_9k_{11}(k_{13}+k_{14})+(k_{10}+k_{11})k_{12}^*k_{14}}\\
     & = \frac{k_{1}k_{3}k_{5}(k_{10} + k_{11})(k_{13} + k_{14})}{k_{1}k_{3}k_{9}k_{11}(k_{13} + k_{14}) + k_{2}(k_{4} + k_{5})(k_{10} + k_{11})k_{12}k_{14}},
      \end{aligned}
    \end{equation*}
    which is the value previously obtained in \cite{ShinarFeinberg2010Robustness} and \cite{karp2012complex}.
\end{example}

\begin{example}
\label{ex:transl_0_1}
Consider the following network:
  \begin{equation}\label{eq:transl_0_1}
  \begin{tikzcd}[column sep=small]
    A \arrow[r, "r_1"] & B \arrow[r, "r_2"] & C, \\[-0.2in]
    2C \arrow[r, "r_3"] & B + C, & \\[-0.2in]
    A + C \arrow[r, rightharpoonup, yshift=0.2ex, "r_4"] &
    D \arrow[l, rightharpoonup, yshift=-0.2ex, "r_5"] \arrow[r, "r_6"] & 2A,
  \end{tikzcd}
  \end{equation}
and a mass action system on this network with kinetic parameters $k = (k_1,
\dots, k_6)$. The results of \cite{ShinarFeinberg2010Robustness} do not identify ACR in any species for the mass action system associated with \eqref{eq:transl_0_1}. The network is not weakly reversible and has a deficiency of $2$. The linear programming approach described in Section \ref{sec:translation_algo} identifies a
  deficiency zero translation which is improper, and a proper deficiency one
  translation. We demonstrate here how the results of Section~\ref{sec:transl_rob} can be
  applied to these translation schemes to identify ACR species and robust values.

\paragraph*{Deficiency one approach:} When restricting to proper translations, the computational program outlined in Section \ref{sec:translation_algo} determines the following translation scheme:
  
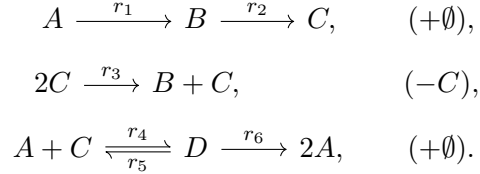
\begin{equation*}\label{eq:transl_1}
  \begin{tikzcd}[column sep=small]
    A \arrow[r, "r_1"] & B \arrow[r, "r_2"] & C, & (+\emptyset), \\[-0.2in]
    2C \arrow[r, "r_3"] & B + C, & & (-C), \\[-0.2in]
    A + C \arrow[r, rightharpoonup, yshift=0.2ex, "r_4"] &
    D \arrow[l, rightharpoonup, yshift=-0.2ex, "r_5"] \arrow[r, "r_6"] & 2A, & (+\emptyset).
  \end{tikzcd}
  \end{equation*}
  This translation scheme identifies the following network:
  
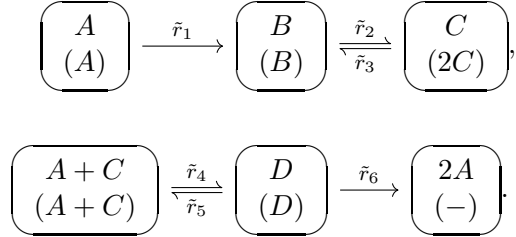
\begin{equation*}
  \begin{tikzcd}[column sep=scriptsize]
    \mbox{\ovalbox{$\begin{array}{c} A \\ (A) \end{array}$}} \arrow[r, "\tilde{r}_1"] & \mbox{\ovalbox{$\begin{array}{c} B \\ (B) \end{array}$}} \arrow[r, rightharpoonup, yshift=+0.2ex, "\tilde{r}_2"] \arrow[r, leftharpoondown, yshift=-0.2ex, "\tilde{r}_3"'] & \mbox{\ovalbox{$\begin{array}{c} C \\ (2C) \end{array}$}}, \\[-0.1in]
    \mbox{\ovalbox{$\begin{array}{c} A+C \\ (A+C) \end{array}$}} \arrow[r, rightharpoonup, yshift=0.2ex, "\tilde{r}_4"] &
    \mbox{\ovalbox{$\begin{array}{c} D \\ (D) \end{array}$}} \arrow[l, rightharpoonup, yshift=-0.2ex, "\tilde{r}_5"]  \arrow[r, "\tilde{r}_6"] & \mbox{\ovalbox{$\begin{array}{c} 2A \\ (-) \end{array}$}}.
  \end{tikzcd}
  \end{equation*}
  Notice that we do not need to assign a kinetic complex to $2A$ since it is not a source complex for any reaction.
  
  This translation is proper and has a stoichiometric deficiency of one.
  Since the complexes $A$, $A+C$, and $D$ are nonterminal,
  we conclude from Theorem~\ref{thm:robustness_translation}, part $2$,
  that these complexes have a robust ratio for the mass action system corresponding to \eqref{eq:transl_0_1}. As a consequence of $A+C$ and $A$ differing in only $C$, the mass action system has ACR in species $C$.
  
  \paragraph*{Deficiency zero approach:} If we do not restrict ourselves to proper translations, the computational program outlined in Section \ref{sec:translation_algo} determines the following translation scheme:
   
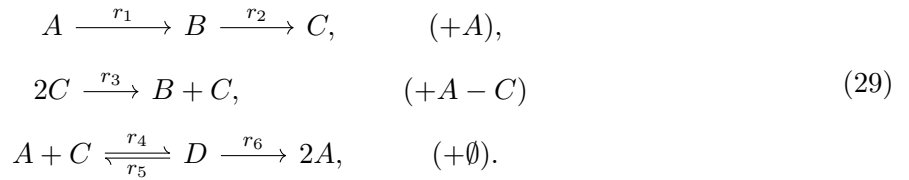
\begin{equation}\label{eq:transl_0}
  \begin{tikzcd}[column sep=small]
    A \arrow[r, "r_1"] & B \arrow[r, "r_2"] & C, & (+A), \\[-0.2in]
    2C \arrow[r, "r_3"] & B + C, & & (+A-C) \\[-0.2in]
    A + C \arrow[r, rightharpoonup, yshift=0.2ex, "r_4"] &
    D \arrow[l, rightharpoonup, yshift=-0.2ex, "r_5"] \arrow[r, "r_6"] & 2A, & (+\emptyset).
  \end{tikzcd}
  \end{equation}
  This gives a network with $5$ complexes, one linkage class, and rank $4$ again, and therefore a deficiency of zero:
    
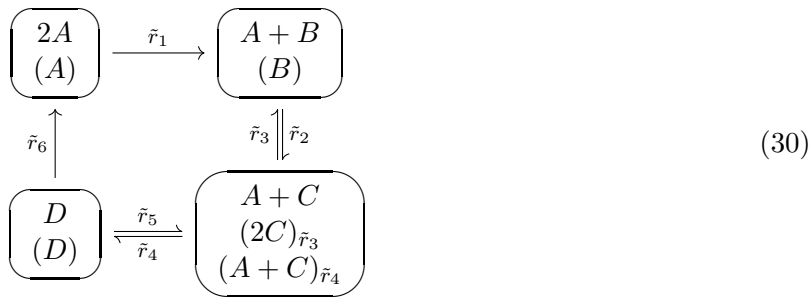
\begin{equation}\label{eq:transl_0_improper}
  \begin{tikzcd}[column sep=normal]
    \mbox{\ovalbox{$\begin{array}{c} 2A \\ (A) \end{array}$}} \arrow[r, "\tilde{r}_1"] & \mbox{\ovalbox{$\begin{array}{c} A+B \\ (B) \end{array}$}} \arrow[d, rightharpoonup, xshift=+0.2ex, "\tilde{r}_2"]\\
    \mbox{\ovalbox{$\begin{array}{c} D \\ (D) \end{array}$}} \arrow[r, rightharpoonup, yshift=0.2ex, "\tilde{r}_5"] \arrow[u, "\tilde{r}_6"] & \mbox{\ovalbox{$\begin{array}{c} A+C \\ (2C)_{\tilde{r}_3} \\ (A+C)_{\tilde{r}_4} \end{array}$}} \arrow[l, rightharpoonup, yshift=-0.2ex, "\tilde{r}_4"] \arrow[u, rightharpoonup, xshift=-0.2ex, "\tilde{r}_3"]
  \end{tikzcd}
  \end{equation}
 The complexes $2C$ and $A+C$ are both translated to the complex $A+C$, i.e.,
 $g(2C) = g(A+C) = A+C$, so that the translation is improper. In order to correspond the translation scheme with a generalized network, we choose $h(A+C) = A+C$. This gives the following generalized network:
    
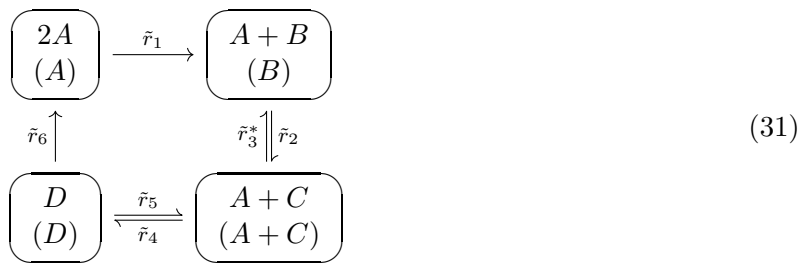
\begin{equation}\label{eq:transl_0_improper_1}
  \begin{tikzcd}[column sep=normal]
    \mbox{\ovalbox{$\begin{array}{c} 2A \\ (A) \end{array}$}} \arrow[r, "\tilde{r}_1"] & \mbox{\ovalbox{$\begin{array}{c} A+B \\ (B) \end{array}$}} \arrow[d, rightharpoonup, xshift=+0.2ex, "\tilde{r}_2"]\\
    \mbox{\ovalbox{$\begin{array}{c} D \\ (D) \end{array}$}} \arrow[r, rightharpoonup, yshift=0.2ex, "\tilde{r}_5"] \arrow[u, "\tilde{r}_6"] & \mbox{\ovalbox{$\begin{array}{c} A+C \\ (A+C) \end{array}$}} \arrow[l, rightharpoonup, yshift=-0.2ex, "\tilde{r}_4"] \arrow[u, rightharpoonup, xshift=-0.2ex, "\tilde{r}_3^*"]
  \end{tikzcd}
  \end{equation}
 where the improperly translated reaction is indicated with an asterisk, i.e. $\mathcal{R}_I = \{ r_3 \}$.

 To enable the application Theorem~\ref{thm:resolvability_translation} part
 1(c), to demonstrate that two complexes have a robust ratio in $\M$, we need to ensure that the corresponding ratio of tree constants in \eqref{eq:transl_0_improper_1} do not depend upon $k_3^*$.
We can determine that
\[\begin{split}
\tilde{K}(2A) & = k_2k_4k_6,\\
\tilde{K}(A+C) & = k_1k_2(k_5+k_6),
\end{split}\]
so that
\[\frac{\tilde{K}(A+C)}{\tilde{K}(2A)} = \frac{k_1(k_5+k_6)}{k_4k_6}.\]
Since this ratio does not depend on $k_3^*$, it follows from Theorem \ref{thm:resolvability_translation} that $A+C$ and $C$ have a robust ratio in $\M$. More precisely, we have
\[x_C = \frac{k_1(k_5+k_6)}{k_4k_6}\]
at every positive steady state of $\M$. Notice that we were able to obtain this result even though we were not able to establish resolvability of the translation $\tN$ in \eqref{eq:transl_0_improper_1}. In fact, the mass action system $\M$ and translated mass action system $\tM$ corresponding to \eqref{eq:transl_0_improper_1} do not have the same positive steady states for any choice of $k_3^*$.
\end{example}

\section{Conclusion}
\label{sec:conclusions}

In this paper we have presented two results, Theorem \ref{thm:robustness_translation} and Theorem \ref{thm:resolvability_translation}, which allow the establishment and computation of robust ratios through a novel application of the process of network translation. Through application of Lemma \ref{lemma1}, the existence of ACR species may also be established. We have shown through several examples that these results expand the scope of networks which can be shown to have robust ratios and ACR through the previous network-based results of \citep{ShinarFeinberg2010Robustness}. We were able to use Theorem \ref{thm:resolvability_translation} to establish the ACR value of the robust species through analysis of the network structure rather than algebraic manipulation of the steady state equations. To aid in application, we have also presented a mixed integer linear programming framework for constructing network translations.

In the future, we plan to investigate the applicability of the methods discussed to the analysis of more complex models, for example of cell signalling pathways or population dynamics. 
Further avenues of research, which would expand the applicability of the theory of translated chemical reaction networks, include:
\begin{enumerate}
\item \emph{More general conditions for resolvability of improperly translated chemical reaction networks.} The strongest results from Theorem \ref{thm:robustness_translation} and Theorem \ref{thm:resolvability_translation} hold for improper and resolvable translations; however, little research has been conducted on methods for establishing resolvability more general than Theorem \ref{thm:resolvability_translation} part $2$. We therefore propose to study more general conditions under which resolvability can be established.
\item \emph{Alternative computational approaches for constructing network translations.} In particular, for large networks the full enumeration of elementary modes adds significantly to the computation time required of the algorithm. We therefore propose to research computational approaches to constructing network translations which do not depend on enumeration of the elementary modes.
  \item \emph{More general methods for corresponding mass action systems to generalized mass action systems.}
    Some mass action systems are dynamically equivalent to weakly reversible generalized mass action networks that cannot be described as translations of the original
    network. We intend to investigate more techniques for the identification of these correspondences, and potentially extend the results of the present and
    previous works~\citep{johnston2014translated,johnston2015computational} to these classes of networks.
\end{enumerate}

\section*{Acknowledgements}
MDJ was supported by the Henry Woodward Fund.
The authors thanks the anonymous reviewers for their helpful comments.

\appendix

\section{Proof of Theorem~\ref{thm:robustness_translation}}\label{app:proofs}

\noindent
  Note that, under the hypotheses of Theorem~\ref{thm:robustness_translation}, we can define a generalized mass action system $\tM$ on $\tN$ that
  is steady state equivalent to $\M$ by either Lemma $2$ or Lemma $4$ of~\cite{johnston2014translated}. Consequently, robust ratios and ACR hold in $\M$ if and only if they hold in $\tM$.
  Theorem~\ref{thm:robustness_translation} is therefore a direct corollary of the following result pertaining to generalized mass action systems. As the proof is nearly identical to that of Lemma S3.20 in the Supplemental Material of \cite{ShinarFeinberg2010Robustness}, we do not restate the background details and results.

\begin{theorem}
\label{thm:robustness_def0_def1_gma}
Consider a generalized chemical reaction network $\N$, with an associated
generalized mass action system $\M$ that admits a positive steady state.
\begin{enumerate}
\item
  If $\N$ has a stoichiometric deficiency of zero and is weakly reversible,
  then $\M$ has a robust ratio in each pair of kinetic complexes $y$ and $y'$ belonging to a common linkage class in $\N$.
\item
  If $\N$ has a stoichiometric deficiency of one,
  then $\M$ has a robust ratio in every pair of nonterminal kinetic complexes $y$ and $y'$ in $\N$.
\end{enumerate}
\end{theorem}
\begin{proof}
  $(1)$ Consider the system of ODEs \eqref{eq:dyn_Ak_gen} corresponding to the generalized mass action system $\M$ and suppose that the network is weakly reversible and has a stoichiometric deficiency of zero. Since the network has a stoichiometric deficiency of zero, it follows that ker$(\mA_\kappa)$=ker$(\mY\mA_\kappa)$.
 Since the network is weakly reversible, ker$(\mY\mA_\kappa)$ may be characterized by Theorem $6$ of~\cite{johnston2014translated}. Furthermore, since the linkage classes provide a partition of ker$(\mY\mA_\kappa)$, by Theorem $3.3$ of~\cite{millan2012chemical} we have that the steady states are generated by binomials of the form
  \[K(y')\vx^{h(y)} - K(y)\vx^{h(y')}, \ y, y' \in \mathcal{L}_\theta, \ \theta = 1, \dots, l,\]
  where the $K(y)$ are the tree constants associated to the kinetic matrix of the
  generalized mass action system.
  As a consequence, at each positive steady state $\bar{\vx} \in \mathbb{R}_{> 0}^n$, and for each pair of complexes $y$,
  $y'$ in the same linkage class, we have
  \[\frac{\bar{\vx}^{h(y)}}{\bar{\vx}^{h(y')}} = \frac{K(y)}{K(y')}\]
  which concludes the proof.\\
  
\noindent (2): Consider the system of ODEs \eqref{eq:dyn_Ak_gen} corresponding to the generalized mass action system $\M$ and suppose that the network has a stoichiometric deficiency of one. From Lemma S3.19 of~\cite{ShinarFeinberg2010Robustness}, ker$(\mA_{\kappa})$
  admits a basis $\{b^1, \ldots, b^t\} \subset \mathbb{R}_{\geq 0}^c$ with support on the $t$ terminal strong
  linkage classes of $(\S, \C, \R)$.
  Lemma S3.20 in the same paper gives $\dim(\ker(\mY \mA_\kappa)) \leq 1 + t$.
  
 Now consider an arbitrary positive equilibrium $\bar{\vx} \in \mathbb{R}_{> 0}^n$. It follows that $\Psi_K(\bar{\vx}) \in \mbox{ker}(\mY \mA_\kappa) \cap \mathbb{R}_{> 0}^c$. Since no vectors in $\{ b^1, \ldots b^t \}$ have support on the nonterminal complexes, we have that $\dim(\ker(\mY \mA_\kappa)) = 1 + t$ so that there is a basis of $\mbox{ker}(\mY \mA_\kappa)$ given by $\{ b^0, b^1, \ldots, b^t\}$ where only $b^0 \in \mathbb{R}_{\geq 0}^c$ has support on the nonterminal complexes. It follows that
 \[\Psi_K(\bar{\vx}) = \lambda_0 b^0 + \sum_{\theta = 1}^t\lambda_\nu b^\theta.\]
  For any nonterminal complexes $y^i, y^j \in \C$, we have that $\bar{\vx}^{h(y^i)} = \lambda_0 b^0_i$ and $\bar{\vx}^{h(y^j)} = \lambda_0 b^0_j$ so that, after solving for $\lambda_0$ and rewriting, we have
  \[\frac{\bar{\vx}^{h(y^i)}}{\bar{\vx}^{h(y^j)}} = \frac{b^0_i}{b^0_j}.\]
  Since the positive equilibrium $\bar{\vx}$ and nonterminal complexes $y^i$ and $y^j$ were chosen arbitrarily, we are done.
\end{proof}

\section{Proof of Theorem~\ref{thm:resolvability_translation}}\label{app:proofs2}

\noindent \emph{Proof of Theorem~\ref{thm:resolvability_translation}}
Assume that $\tN$ is a translation of $\N$ which has a stoichiometric deficiency of zero and is weakly reversible, and that $\tM$ is the associated translated mass action system with rate constant vector $\tilde{k} \in \mathbb{R}_{> 0}^{\tilde{m}}$ defined according to \eqref{eq:translation_kin_par}. We start by rewriting the dynamical equations \eqref{eq:dynamics} associated with $\M$ in a manner which emphasizes the kinetic complexes of the translation, then considering the steady equations of the mass action system $\M$ when rewritten in this manner.

By considering reactions $r \in \mathcal{R}_I$ and $r \notin \mathcal{R}_I$ separately, the system \eqref{eq:dynamics} associated with $\M$ may be written
\begin{equation}\label{eq:sseqs}
  \begin{aligned}
  \frac{d\mathbf{x}}{dt} & = \sum_{r: y \rightarrow y' \in \mathcal{R}}  k_r (y' - y) \; \vx^y,
  \\ & = \sum_{\substack{r: y \rightarrow y' \in \mathcal{R} \\ r \notin \R_I}}  k_r (y' - y) \; \vx^{y} + \sum_{\substack{r: y \rightarrow y' \in \mathcal{R} \\ r\in \R_I}}  k_r (y' - y) \; \vx^{y},\\
  & = \sum_{\substack{r: y \rightarrow y' \in \mathcal{R} \\ r\notin \R_I}} k_r (y' - y) \; \vx^{y} + \sum_{\substack{r: y \rightarrow y' \in \mathcal{R} \\ r\in \R_I}} k_r(\mathbf{x}) (y' - y) \; \vx^{h(g(y))}.
  \end{aligned}
\end{equation}
where $k_r(\mathbf{x}) := \left( \vx^{y}/\vx^{h(g(y))} \right) k_r$ for $r: y \to y' \in \mathcal{R}_I$. This formulation emphasizes the dependence of \eqref{eq:dynamics} on the kinetic complexes of the translation, at the expense of having state-dependent rates $k_r(\mathbf{x})$ for $r \in \mathcal{R}_I$.

Denote by $\tilde{\mY}$ the matrix corresponding to the stoichiometric complexes of the translated reaction network,
and by $\tilde{\Psi}_K(\vx) = (\vx^{h(\tilde{y}^1)}, \dots, \vx^{h(\tilde{y}^{\tilde{c}})})$ the vector of monomials of the kinetic complexes. Denote by $\tilde{\mA}_\kappa(\vx)$ the kinetic matrix associated to the mass action system on the translated chemical reaction network with kinetic parameters defined by
\begin{equation}
\label{eq:adjusted}
  \tilde{k}_{\tilde{r}}(\mathbf{x}) = \sum_{\substack{r: y \to y' \in \R \setminus \mathcal{R}_I\\f(r) =\tilde{r}}}k_{r} + \sum_{\substack{r: y \to y' \in \mathcal{R}_I\\f(r) = \tilde{r}}} k_r(\mathbf{x}).
\end{equation}
We can then rewrite the dynamical equations \eqref{eq:sseqs} associated with $\M$ as
\begin{equation}
\label{eq:sseqs3}
\frac{d\mathbf{x}}{dt} = \tilde{\mY} \tilde{\mA}_{\kappa}(\vx) \tilde{\Psi}_K(\vx).
\end{equation}
On the other hand, the dynamical equations associated with $\tM$ can be written as
\begin{equation}
\label{eq:sseqs3bis}
\frac{d\mathbf{x}}{dt} = \tilde{\mY} \tilde{\mA}_{\kappa} \tilde{\Psi}_K(\vx),
\end{equation}
where the kinetic matrix $\tilde{\mA}_{\kappa}$ is defined using the kinetic parameters in~\eqref{eq:translation_kin_par}.
Hence the two systems~\eqref{eq:sseqs3} and~\eqref{eq:sseqs3bis} differ only in the rates corresponding to reactions $r \in \mathcal{R}_I$.
Recall that we write $\tilde{K}(y)$, $\tilde{y} \in \tC$, for the tree constants associated to $\tilde{\mA}_\kappa$.

Now suppose that $\bar{\vx} \in \mathbb{R}_{> 0}^n$ is a positive steady state for $\M$. From \eqref{eq:sseqs} and \eqref{eq:adjusted}, and noting that the stoichiometric deficiency of the translated reaction network is zero by assumption, we have that
\begin{equation}\label{eq:sseqs2}
  \begin{aligned}
    \tilde{\mY}\tilde{\mA}_\kappa(\bar{\vx})\tilde{\Psi}_K(\bar{\vx}) = \textbf{0} \; \Longleftrightarrow \; \tilde{\mA}_{\kappa}(\bar{\vx})\tilde{\Psi}_K(\bar{\vx}) = \mathbf{0}.
  \end{aligned}
\end{equation}
We define $K(\tilde{y}^i;\mathbf{x})$ to be the tree constant \eqref{def:tc} associated with the stoichiometric complex $\tilde{y}^i$ in the translated reaction network with the state-dependent rates \eqref{eq:adjusted}. By Lemma S3.19 of \citet{ShinarFeinberg2010Robustness}, $\mbox{ker}(\tilde{\mA}_\kappa(\bar{\vx}))$ admits a basis with support on the terminal strong linkage classes of the network, which coincide with the linkage classes of the translated reaction network by weak reversibility. By Theorem $6$ of~\cite{johnston2014translated}, this basis has the form $\{ \tilde{\mathbf{K}}^1(\bar{\vx}), \ldots, \tilde{\mathbf{K}}^{\tilde{l}}(\bar{\vx}) \} \subset \mathbb{R}_{\geq 0}^{\tilde{c}}$ where $\tilde{\mathcal{L}}_\theta$, $\theta = 1, \ldots, \tilde{l},$ are the linkage classes of $\tN$, and $\tilde{\mathbf{K}}^\theta_j(\bar{\vx}) = \tilde{K}(y^j;\bar{\vx})$ if $y^j \in \tilde{\mathcal{L}}_\theta$ and $\tilde{\mathbf{K}}^\theta_j (\bar{\vx})= 0$ otherwise.\\

\noindent \emph{Proof (1).} Consider complexes $y$, $y'$ and $\tilde{y}$, $\tilde{y}'$ as in point $(1)$.
Since the complexes $\tilde{y}$, $\tilde{y}'$ belong to the same linkage class and $y=h(\tilde{y})$, $y'=h(\tilde{y}')$, we have
\begin{equation}
\label{eq:ratio}
\frac{\bar{\vx}^y}{\bar{\vx}^{y'}} = \frac{\bar{\vx}^{h(\tilde{y})}}{\bar{\vx}^{h(\tilde{y}')}} = \frac{\tilde{K}(\tilde{y};\bar{\vx})}{\tilde{K}(\tilde{y}';\bar{\vx})}.
\end{equation}
Notice that the tree constants of $K(y^i;\bar{\vx})$ of $\tilde{\mA}_\kappa(\bar{\vx})$ and $K(y^i)$ of $\tilde{\mA}_\kappa$ are given by the same polynomials in the kinetic parameters. Consequently, their values differ only in the corresponding entries $k_r(\bar{\vx})$ and $k_r^*$ for $r \in \mathcal{R}_I$.\\

\noindent \emph{(a)} If the translation is proper then $\mathcal{R}_I = \emptyset$ and the kinetic matrices $\tilde{\mA}_\kappa(\vx)$ and
  $\tilde{\mA}_\kappa$ coincide. It follows trivially that the ratio \eqref{eq:ratio} does not depend upon $\bar{\vx}$ and reduces to $\tilde{K}(\tilde{y})/\tilde{K}(\tilde{y}')$.\\

\noindent \emph{(b)} Suppose that the translation is improper and resolvable with the rate constants $k^*_r$, $r \in \mathcal{R}_I$, in \eqref{eq:translation_kin_par} chosen such that $\M$ and $\tM$ are steady state equivalent. Consequently, we have that $\tilde{\mA}_\kappa\tilde{\Psi}_K(\bar{\vx})=0$. It follows that, for this choice of rate constants, the ratio \eqref{eq:ratio} does not depend on $\bar{\vx}$ and resolves to $\tilde{K}(\tilde{y})/\tilde{K}(\tilde{y}')$.\\

\noindent \emph{(b)} 
If the translation is improper, then the systems~\eqref{eq:sseqs3} and~\eqref{eq:sseqs3bis} do not necessarily admit the same positive steady states. However, if the ratio \eqref{eq:ratio} does not depend on any rate
constant $k_r$, $r \in \R_I$, then in particular it does not depend on $\bar{\vx}$, and is a ratio of polynomials in the $k_r$, $r \notin \R_I$, that coincides with $\tilde{K}(\tilde{y})/\tilde{K}(\tilde{y}')$.\\

\noindent \emph{Proof (2).} \emph{(a)} Suppose the pair of complexes $y$, $y'$ are translated to the same complex, and
\begin{equation}\label{eq:lin_comb}
  y - y' = \sum_{\theta=1}^{\tilde{l}} \mathop{\sum_{i,j=1}^{\tilde{c}}}_{\tilde{y}^i, \tilde{y}^j \in \tilde{\mathcal{L}}_\theta} c_{ij} (h(\tilde{y}^{i}) - h(\tilde{y}^{j}))
\end{equation}
for some constants $c_{ij} \in \mathbb{R}$.
Then, for any positive steady state $\bar{\vx} \in \mathbb{R}_{> 0}^n$ of $\M$, we can write
\[\frac{\bar{\vx}^y}{\bar{\vx}^{y'}} =\prod_{\theta=1}^{\tilde{l}} \mathop{\prod_{i,j=1}^{\tilde{c}}}_{y_i, y_j \in \mathcal{L}_\theta}\left( \frac{\bar{\vx}^{h(\tilde{y}^i)}}{\bar{\vx}^{h(\tilde{y}^j)}} \right)^{c_{ij}}.\]
As in part (1), since $\Psi_K(\bar{\vx}) \in \mbox{ker}(\tilde{\mA}_\kappa(\bar{\vx}))$, we have that at steady state 
\[ \frac{\bar{\vx}^{h(\tilde{y}^i)}}{\bar{\vx}^{h(\tilde{y}^j)}} = \frac{\tilde{K}(\tilde{y}^i;\bar{\vx})}{\tilde{K}(\tilde{y}^j;\bar{\vx})} ,\]
and as a consequence
\begin{equation}\label{eq:ratio_improper}
\frac{\bar{\vx}^y}{\bar{\vx}^{y'}} = \prod_{\theta=1}^{\tilde{l}} \mathop{\prod_{i,j=1}^{\tilde{c}}}_{y_i, y_j \in \mathcal{L}_\theta} \left(
    \frac{\tilde{K}(\tilde{y}^i;\bar{\vx})}{\tilde{K}(\tilde{y}^j;\bar{\vx})} \right)^{c_{ij}}.
\end{equation}
By hypothesis, the kinetic adjustment factor $\tilde{K}(y,y')$
does not depend on any rate constant $k_r$ corresponding to a reaction $r \in \mathcal{R}_I$. Since the only state-dependent rates in \eqref{eq:ratio_improper} are those corresponding to reactions $r \in \mathcal{R}_I$, we have that the ratio \eqref{eq:ratio_improper} coincides with the kinetic adjustment factor $\tilde{K}(y,y')$ defined by \ref{def:kin_adj_factor}. It follows that and $y$ and $y'$ have a robust ratio in $\M$ and, specifically, that $\mathbf{x}^{y}/\mathbf{x}^{y'} = \tilde{K}(y,y')$ at every positive steady state $\mathbf{x} \in \mathbb{R}_{> 0}^n$.\\

\noindent \emph{(b)} Now suppose that the conditions of 2 hold for all pairs of complexes $y, y'$ such that $g(y)=g(y')$. Using~\eqref{eq:lin_comb} with $y'=h(g(y))$, we can write
\[y - h(g(y)) = \sum_{\theta=1}^{\tilde{l}} \mathop{\sum_{i,j=1}^{\tilde{c}}}_{\tilde{y}^i, \tilde{y}^j \in \tilde{\mathcal{L}}_\theta}c_{ij} (h(\tilde{y}^{i}) - h(\tilde{y}^{j})).\]
Let $\bar{\vx} \in \mathbb{R}_{> 0}^n$ be a positive steady state of $\M$. From part (a), we have that $\bar{\vx}^y/\bar{\vx}^{h(g(y))} = \tilde{K}(y,h(g(y)))$. 
Furthermore, by \eqref{eq:sseqs}, we have that
\begin{equation}\label{eq:ssMtM}
  \begin{aligned}
  \mathbf{0} =& \sum_{\substack{r: y \rightarrow y' \in \mathcal{R} \\ r\notin \R_I}} k_r (y' - y) \; \bar{\vx}^{y} + \sum_{\substack{r: y \rightarrow y' \in \mathcal{R} \\ r\in \R_I}} k_r(\bar{\vx}) (y' - y) \; \bar{\vx}^{h(g(y))}\\
             = & \sum_{\substack{r: y \rightarrow y' \in \mathcal{R} \\ r\notin \R_I}} k_r (y' - y) \; \bar{\vx}^{y} + \sum_{\substack{r: y \rightarrow y' \in \mathcal{R} \\ r\in \R_I}} \left[\tilde{K}(y,h(g(y)))k_r\right] (y' - y) \; \bar{\vx}^{h(g(y))}.
  \end{aligned}
\end{equation}
It follows that $\bar{\vx}$ is a steady state of $\tM$.
Conversely, if $\bar{\vx} \in \mathbb{R}_{> 0}^n$ is a positive steady state of $\tM$, since $\bar{\vx}^{y}/\bar{\vx}^{h(g(y))} = \tilde{K}(y,h(g(y)))$, from~\eqref{eq:ssMtM} we can conclude that $\bar{\vx}$ is also a positive steady state for $\M$. It follows that $\M$ and $\tM$ are steady state equivalent, and consequently the translation is resolvable.
  {\hfill $\square$}

\section{Details of the Mixed-Integer Linear Programming code}\label{app:code}

In this appendix, we describe in more detail the method sketched in Section~\ref{sec:translation_algo}.
We first fix some notations.
We define $\mathbb{B} = \{0,1\}$, and denote by $\mI^+$ and $\mI^-$ the positive and negative part of the incidence matrix, respectively, and set $\mGamma^+=\mY\mI^+$, $\mGamma^-=\mY\mI^-\in \mathbb{N}^{n \times m}$.
In other words, the columns of $\mGamma^+$ and $\mGamma^-$ contain the stoichiometric coefficients of the product and reactant complexes of the network, respectively.
Given $\alpha$ cycles and $\beta$ stoichiometric elementary modes, we denote by $I^h \subseteq \{1,\dots,m\}$,
$h=1,\dots,\alpha$, the indices of reactions in the cyclic elementary modes, and by $J^h \subseteq \{1,\dots,m\}$, $h=1,\dots,\beta$, the indices of reactions in the stoichiometric elementary modes.\\

\noindent \emph{Translation complexes are well defined:} We define a matrix of decision variables $\mUpsilon \in \mathbb{R}^{n \times m}$ where the $j^{th}$ column of $\mUpsilon$ corresponds to the translation complex $\upsilon^j$ of the reaction $r_j$, i.e. $\mUpsilon_{\cdot j} = \upsilon^j$.
Since each complex in the original network can be translated to only one
stoichiometric complex in the translated network, any two reactions $r_{j_1}$ and $r_{j_2}$ that have the same reactant need to be translated by the same translation complex:
\begin{equation}\label{c:same_reactant}\tag{React}
  \mGamma^-_{i j_1}=\mGamma^-_{i j_2} \mbox{ for all } i=1,\dots,n \Rightarrow \mUpsilon_{ij_1} = \mUpsilon_{ij_2} \mbox{ for all } i=1,\dots,n.
\end{equation}
We can optionally require that all the translation complexes are positive, by imposing
\begin{equation}\label{c:pos-t}\tag{PosT}
  \mUpsilon_{ij}\geq 0, i=1,\dots,n, j=1,\dots,m.
\end{equation}
If we do not impose that the translation complexes have positive coefficients,
we need to guarantee that all the coefficients of the complexes resulting from the translations
are positive. To this end, we require
\begin{equation}\label{c:pos-c}\tag{PosC}
  \mUpsilon_{ij}+\mGamma^-_{ij}\geq 0, \mUpsilon_{ij}+\mGamma^+_{ij}\geq 0, i=1,\dots,n, j=1,\dots,m.
\end{equation}

\noindent \emph{Stoichiometric elementary modes translated to cycles:}
Suppose now that the reactions $r_{j_1}, \dots, r_{j_k}$ define a cycle.
To ensure that these reactions form a cycle in the translated network as well, we impose
\begin{equation}\label{c:cycle}\tag{Cycle}
  \mUpsilon_{ij_1} = \mUpsilon_{ij_2} = \cdots = \mUpsilon_{ij_k}, \ i=1,\dots,n.
\end{equation}
Consider now the stoichiometric elementary modes. Suppose that the reactions $r_{j_1}, \dots, r_{j_k}$ define the $h^{th}$ stoichiometric elementary mode, and $\hat{\upsilon}^{j_1}, \dots, \hat{\upsilon}^{j_k}$ are complexes such that the reactions $r_{j_1} + \hat{\upsilon}^{j_1}, \dots, r_{j_k} + \hat{\upsilon}^{j_k}$ define a cycle by Lemma \ref{lem:ems_to_cycles}, for a fixed order of the reactions. We denote by $\hat{\mUpsilon}^h \in \mathbb{Z}^{n \times |J^h|}$ the matrix with columns the complexes $\hat{\upsilon}^{j_1},\dots,\hat{\upsilon}^{j_k}$.
Write $j'$ for the position of the index $j$ in $J^h$, i.e., $j=J^h_{j'}$.
Then to guarantee that $r_{j_1}, \dots, r_{j_k}$ are translated to the identified cycle we impose
\begin{equation}\label{c:stoich}\tag{Stoich}
\mUpsilon_{ij_1} - \hat{\mUpsilon}^h_{ij'_1} = \mUpsilon_{ij_2} - \hat{\mUpsilon}^h_{ij'_2} = \cdots = \mUpsilon_{ij_k} - \hat{\mUpsilon}^h_{ij'_k}, \mbox{ for } h = 1 , \ldots, \beta, i = 1, \dots, n.
\end{equation}
For each stoichiometric elementary mode we have therefore a set of constraints of the form~\eqref{c:stoich}.
These constraints might not be satisfiable at the same time for all stoichiometric elementary modes; hence we want to maximize the number of
stoichiometric elementary modes for which these constraints are verified. We do so by introducing additional variables.\\

\noindent \emph{Minimize stoichiometric modes in translated network:}
For the $h^{th}$ stoichiometric elementary mode, we introduce a binary variable $\sigma_h$ that will be equal to $1$ if the elementary mode is not translated to a cycle. 
The restrictions on $\sigma_h$ are obtained by imposing, for each pair of indices $j_1,j_2\in J^h$ and for each species $X_i$
the constraint
\begin{equation}\label{c:count}\tag{Count}
  \sigma_h \geq \varepsilon (\mUpsilon_{ij_1} - \hat{\mUpsilon}^h_{ij'_1} - \mUpsilon_{ij_2} + \hat{\mUpsilon}^h_{ij'_2}).
\end{equation}
Notice that \eqref{c:count} reduces to \eqref{c:stoich} if $\sigma_h = 0$. In order to maximize the number of stoichiometric modes which are translated to cycles, we introduce the following objective function:
\begin{equation}\label{obj:max_transl}\tag{Obj}
\mathrm{minimize} \; \; \sum_{h=1}^{\beta} \sigma_h.
\end{equation}
\vspace{0.05in}

\noindent \emph{Permutations of stoichiometric modes (optional):} As observed in Remark~\ref{rmk:order_reacts}, we can identify a possible cycle for each possible order of the reactions involved in the elementary mode.
In general, the existence of a solution can depend on the choice of the order. We take the order of the reactions into account by introducing variables and constraints to keep track of the possible permutations of the reactions in the elementary mode.
In addition, since the order of reactions in the elementary mode is not fixed, instead of calculating the translations that convert the elementary mode to a cycle,
we consider the matrices $\hat{\mUpsilon}^h\in \mathbb{Z}^{n \times |J^h|}$ as matrices of decision variables, and impose the constraint \eqref{c:count} as done previously in the
case of elementary modes with a fixed order.

To keep all the possible orders into account, for the $h^{th}$ stoichiometric elementary mode, with set of reaction indices $J^h$, we define $|J^h| \times |J^h|$ binary
variables $\mP^h \in \mathbb{B}^{|J^h| \times |J^h|}$, that will identify the position of
each reaction in the possible orders: $\mP^h_{tj}$ will be equal to $1$ if and only
if the $j^{th}$ reaction is in position $t$.

We have therefore the following constraints. To ensure that each reaction is assigned one and only one position, we impose that each row and column of $\mP^h$ sums to $1$:
\begin{equation}\label{c:perm1}\tag{Perm1}
	\sum_{j=1}^k \mP^h_{tj} = 1 \ \text{for all } t = 1, \dots, |J^h|, \\
\end{equation}
\begin{equation}\label{c:perm2}\tag{Perm2}
	\sum_{t=1}^k \mP^h_{tj} = 1 \ \text{for all } j = 1, \dots, |J^h|.
\end{equation}

Now, suppose that $j'_1$, $j'_2$, $t$ and $t'$ are indices in $\{1,\dots,|J^h|\}$ such that $t'=t+1\Mod{|J^h|}$.
Write $j_1$, $j_2$ for $J^h_{j'_1}$ and $J^h_{j'_2}$, respectively.
If $\mP^h_{tj'_1}$ and $\mP^h_{tj'_2}$ are both equal to $1$,
then reaction $r_{j_1}$ is followed by reaction $r_{j_2}$ in the cycle.
If this is the case, we want to impose that the product of reaction $r_{j_1}$ is equal to the reactant of reaction $r_{j_2}$.
For each pair of indices $J^h_{j'_1}$, $J^h_{j'_2}$ in the elementary mode we therefore add the following constraints, for each species $i=1,\dots,n$:
\begin{equation}\label{c:perm3}\tag{Perm3}
   -\frac{1}{\varepsilon} (2 - \mP^h_{tj'_1} - \mP^h_{t'j'_2}) \leq \mGamma^+_{ij_1}-\mGamma^-_{ij_2} + \hat{\mUpsilon}^h_{ij'_1}-\hat{\mUpsilon}^h_{ij'_2}\leq \frac{1}{\varepsilon} (2 - \mP^h_{tj'_1} - \mP^h_{t'j'_2}), \ j' = j+1\Mod{|J^h|}.\\
\end{equation}
We can also fix the position of one reaction arbitrarily, and set
\begin{equation}\label{c:perm4}\tag{Perm4}
	\mP^h_{00}= 1 \ \text{for all } h = 1, \dots, \beta. \\
\end{equation}
\vspace{0.05in}

\noindent \emph{Proper translations (optional):} We might sometimes be interested in identifying a proper translation with
minimum deficiency.
To impose that distinct complexes in the original network are translated to
distinct complexes in the translated network, we introduce $n \times m \times m$
binary variables $\mU_{ij_1j_2}$.
The variable $\mU_{ij_1j_2}$ encodes whether the reactant complexes of reactions $r_{j_1}$ and $r_{j_2}$ have been
translated to the same complex with respect to species $X_i$, i.e.
\[|\mGamma^-_{ij_1} + \mUpsilon_{ij_1} - \mGamma^-_{ij_2} - \mUpsilon_{ij_2}| = 0 \ \Rightarrow \ \mU_{ij_1j_2} = 1.\]
This is achieved by introducing additional $n \times m \times m$ auxiliary binary
variables $\mV_{ij_1j_2}$, and considering the constraints
\begin{equation}\label{c:proper1}\tag{Proper1}
\left\{ \begin{split} \mGamma^-_{ij_1} + \mUpsilon_{ij_1} - \mGamma^-_{ij_2} - \mUpsilon_{ij_2} & \geq \epsilon (1 - \mU_{ij_1j_2}) - M \mV_{ij_1j_2}, \\ \mGamma^-_{ij_1} + \mUpsilon_{ij_1} - \mGamma^-_{ij_2} - \mUpsilon_{ij_2} &\leq - \epsilon (1 - \mU_{ij_1j_2}) + M(1-\mV_{ij_1j_2}),\end{split} \right.
\end{equation}
as can be easily checked by considering the four possible cases.
Finally, we have to impose that, if the reactants of reactions $r_{j_1}$ and $r_{j_2}$ differ, then
their corresponding translated complexes differ in at least one species. Since the variables
$\mU_{ij_1j_2}$ count the number of matching species, we impose, for $j_1,j_2=1,\dots,m$
\begin{equation}\label{c:proper2}\tag{Proper2}
\mGamma^-_{ij_1} \neq \mGamma^-_{ij_2} \text{ for any } i=1,\dots,n \ \Rightarrow \sum_{i=1}^{n} \mU_{ij_1j_2} \leq n - 1.
\end{equation}
The variables $\mU_{ij_1j_2}$ and $\mV_{ij_1j_2}$ corresponding to reactions $r_{j_1}$ and $r_{j_2}$ with the same reactant can be omitted.
The parameters, variables and constraints of the problem are summarized in the following tables.

\subsubsection*{Parameters}
\begin{tabular}{lp{12cm}}
    $n \in \mathbb{N}$ & number of species \\
    $c \in \mathbb{N}$ & number of complexes \\
    $m \in \mathbb{N}$ & number of reactions \\
    $\mGamma \in \mathbb{N}^{n \times m}$ & stoichiometric matrix \\
    $\mGamma^- \in \mathbb{N}^{n \times m}$ & stoichiometric coefficients of the reactants \\
    $\mGamma^+ \in \mathbb{N}^{n \times m}$ & stoichiometric coefficients of the products \\
    $\alpha \in \mathbb{N}$ & number of cyclic elementary modes \\
    $\beta \in \mathbb{N}$ & number of stoichiometric elementary modes \\
    $I^h \subseteq \{1,\dots,m\}$, $h=1,\dots,\alpha$ & indices of reactions in cycles \\
    $J^h \subseteq \{1,\dots,m\}$, $h=1,\dots,\beta$ & indices of reactions in stoichiometric elementary modes \\
    $M \in \mathbb{R}_{>0}$ & $M \gg 1$ \\
    $\epsilon \in \mathbb{R}_{>0}$ & $0 < \epsilon \ll 1$ \\
\end{tabular}

\subsubsection*{Decision Variables}
{\renewcommand{\arraystretch}{1.2}
\begin{tabular}{lp{12cm}}
  $\mUpsilon \in \mathbb{R}^{n \times m}$ &
  matrix of translation complexes: $\mUpsilon_{ij} = \upsilon^j_i$ is the stoichiometric coefficient of species $X_i$ in the translation of the reaction $r_j$\\
  $\hat{\mUpsilon}^h \in \mathbb{Z}^{n \times |J^h|}$, $h=1,\dots,\beta$ &
  matrix of translation complexes which convert elementary mode $h$ to cycle: $\hat{\mUpsilon}^h_{ij} = \hat{\upsilon}^j_i$ is the stoichiometric coefficient of species $X_i$ in the translation of the $j^{th}$ reaction in the $h^{th}$ elementary mode to a cycle. $\hat{\mUpsilon}^h$ is a matrix of variables if permutations of reactions are considered, and is otherwise calculated using Lemma~\ref{lem:ems_to_cycles}. \\
  $\mP^h \in \mathbb{B}^{|J^h| \times |J^h|}$, $h=1,\dots,\beta$ &
  orders of reactions in elementary mode $h$: $\mP^h_{tj}=1$ iff the $j^{th}$ reaction in the $h^{th}$ stoichiometric elementary mode is in position $t$\\
  $\sigma \in \mathbb{B}^{\beta}$ &
  count the elementary modes that are translated to cycles: $\sigma_h=1$ iff the $h^{th}$ stoichiometric elementary mode is translated to a cycle\\
  $\mU \in \mathbb{B}^{n\times m\times m}$ &
  (optional for proper) count the number of matching species in the translation of the reactants:
  $\mU_{ij_1j_2}=1$ iff the translations of the reactants of reactions $r_{j_1}$ and $r_{j_2}$ are equal in species $i$ \\
    $\mV \in \mathbb{B}^{n\times m\times m}$ & (optional for proper) auxiliary variables \\
\end{tabular}
}

\subsubsection*{Objective}
{\renewcommand{\arraystretch}{1.4}
\begin{tabular}{lp{12cm}}
$\mathrm{minimize} \ \sum_{h=1}^{\beta}\sigma_{h}$ & Minimize stoichiometric modes in translated network \textbf{(\hypertarget{obj:max_transl}{Obj})}\\
\end{tabular}
}

\subsubsection*{Constraints}

{\renewcommand{\arraystretch}{1.4}
  \begin{align*}
&\left\{\begin{tabular}{p{10cm}p{7cm}}
  $\mGamma^-_{i j_1}=\mGamma^-_{i j_2}, \forall i=1,\dots,n \Rightarrow \mUpsilon_{ij_1}=\mUpsilon_{ij_2}, \forall i=1,\dots,n$, & \multirow{2}{7cm}{reactions with the same reactant are translated by the same complex \textbf{(\hypertarget{c:same_reactant}{React})}} \\
  $j_1, j_2=1,\dots,m, j_1\neq j_2$ & 
\end{tabular}\right.\\
&\left\{\begin{tabular}{p{10cm}p{7cm}}
  $\mUpsilon_{ij}\geq 0$, $i=1,\dots,n$, $j=1,\dots,m$& translations are positive (optional) \textbf{(\hypertarget{c:pos-t}{PosT})}\\
    $\mUpsilon_{ij}+\mGamma^-_{ij}\geq 0$, $\mUpsilon_{ij}+\mGamma^+_{ij}\geq 0$, $i=1,\dots,n$, $j=1,\dots,m$&
    all resulting stoichiometric coefficients are positive \textbf{(\hypertarget{c:pos-c}{PosC})}\\
\end{tabular}\right.\\
&\left\{\begin{tabular}{p{10cm}p{7cm}}
  $\mUpsilon_{ij_1}=\mUpsilon_{ij_2}$, $i=1,\dots,n$, & \\
  $j_1=I^h_{j}$, $j_2=I^h_{j+1}$, $j=1,\dots,|I^h|-1$, & cycles are preserved \textbf{(\hypertarget{c:cycle}{Cycle})} \\
  $h=1,\dots, \alpha$ &
\end{tabular}\right.\\
&\left\{\begin{tabular}{p{10cm}p{7cm}}
    $\left\{\begin{array}{l}
    \sigma_h \geq \varepsilon (\mUpsilon_{ij_1} - \hat{\mUpsilon}^h_{ij'_1} - \mUpsilon_{ij_2} + \hat{\mUpsilon}^h_{ij'_2}),\\
     i=1,\dots,n, j_1=J_{j'_1}^h, j_2=J_{j'_2}^h,\\
    j'_1=1,\dots,|J^h|, j'_2=1,\dots,|J^h|, j'_1\neq j'_2\\
    \end{array}\right.$ & \vspace{-25pt}$\sigma_h=1$ if the $h^{th}$ stoichiometric elementary mode is not translated to a cycle \textbf{(\hypertarget{c:count}{Count})}\\
  $\displaystyle{\sum_{j=1}^{|J^h|}}\mP_{tj}^h=1$, $t=1,\dots,|J^h|$ & each index in the order corresponds to only one reaction \textbf{(\hypertarget{c:perm1}{Perm1})}\\
  $\displaystyle{\sum_{t=1}^{|J^h|}}\mP_{tj}^h=1$, $j=1,\dots,|J^h|$ & each reaction is assigned only one index in the order \textbf{(\hypertarget{c:perm2}{Perm2})}\\
    $\left\{\begin{array}{l}
    \mGamma^+_{ij_1}-\mGamma^-_{ij_2} + \hat{\mUpsilon}^h_{ij'_1}-\hat{\mUpsilon}^h_{ij'_2}\leq \frac{1}{\varepsilon} (2 - \mP^h_{tj'_1} - \mP^h_{t'j'_2}),\\
    \mGamma^+_{ij_1}-\mGamma^-_{ij_2} + \hat{\mUpsilon}^h_{ij'_1}-\hat{\mUpsilon}^h_{ij'_2}\geq -\frac{1}{\varepsilon} (2 - \mP^h_{tj'_1} - \mP^h_{t'j'_2}),\\
    t=1,\dots,|J^h|,t' = t+1\Mod{|J^h|},\\
     i=1,\dots,n, j_1=J^h_{j'_1}, j_2=J^h_{j'_2},\\
    j'_1=1,\dots,|J^h|, j'_2=1,\dots,|J^h|, j'_1\neq j'_2
    \end{array}\right.$ &
    \vspace{-20pt}
  the product of the reaction with index $t$ is aligned with the reactant of reaction with index $t+1\Mod{|J^h|}$ \textbf{(\hypertarget{c:perm3}{Perm3})}\\
  $\mP^h_{00}=1$ & first reaction is in first position \textbf{(\hypertarget{c:perm4}{Perm4})}\\
  $h=1,\dots,\beta$ &
\end{tabular}\right.\\
&\left\{\begin{tabular}{p{10cm}p{7cm}}
  $\left\{\begin{array}{l}
  \mGamma^-_{ij_1} + \mUpsilon_{ij_1} - \mGamma^-_{ij_2} - \mUpsilon_{ij_2} \geq \epsilon (1 - \mU_{ij_1j_2}) - M \mV_{ij_1j_2},\\
  \mGamma^-_{ij_1} + \mUpsilon_{ij_1} - \mGamma^-_{ij_2} - \mUpsilon_{ij_2} \leq - \epsilon (1 - \mU_{ij_1j_2}) + M(1-\mV_{ij_1j_2}),\\
  i=1,\dots,n
  \end{array}\right.$
  & {proper translation (optional) \textbf{(\hypertarget{c:proper1}{Proper1})}}\\
  $\mGamma^-_{ij_1} \neq \mGamma^-_{ij_2} \text{ for any } i=1,\dots,n \ \Rightarrow \sum_{i=1}^{n} \mU_{ij_1j_2} \leq n - 1$ & \textbf{(\hypertarget{c:proper2}{Proper2})}\\
  $j_1,j_2=1,\dots,m$. &
\end{tabular}\right.
 \end{align*}}

\begin{thebibliography}{24}
\providecommand{\natexlab}[1]{#1}
\providecommand{\url}[1]{\texttt{#1}}
\expandafter\ifx\csname urlstyle\endcsname\relax
  \providecommand{\doi}[1]{doi: #1}\else
  \providecommand{\doi}{doi: \begingroup \urlstyle{rm}\Url}\fi

\bibitem[Alon et~al.(1999)Alon, Surette, Barkai, and Leibler]{Alon1999}
Alon U, Surette MG, Barkai N, Leibler S (1999)
\newblock Robustness in bacterial chemotaxis.
\newblock Nature 397\penalty0(6715):\penalty0 168--171.

\bibitem[Anderson et~al.(2014)Anderson, Enciso, and Johnston]{A-E-J}
Anderson DF, Enciso G, Johnston MD (2014)
\newblock Stochastic analysis of chemical reaction networks with absolute
  concentration robustness.
\newblock J. R. Soc. Interface, 11\penalty0 (93):\penalty0 20130943.

\bibitem[Anderson et~al.(2017)Anderson, Cappelletti, and Kurtz]{Anderson2016}
Anderson DF, Cappelletti D, Kurtz TG (2017)
\newblock Finite time distributions of stochastically modeled chemical systems
  with absolute concentration robustness.
\newblock SIAM J. Appl. Dyn. Syst., 16\penalty0 (3):\penalty0 1309--1339.

\bibitem[Brauer and Castillo(2012)Brauer and Castillo]{brauer2012}
  Brauer F, Castillo-Chavez C (2012)
  \newblock Mathematical Models in Population Biology and Epidemiology.
  \newblock New York: Springer (Vol. 1).

\bibitem[Craciun et~al.(2009)Craciun, Dickenstein, Shiu, and
  Sturmfels]{craciun2009toric}
Craciun G, Dickenstein A, Shiu A, Sturmfels B (2009)
\newblock Toric dynamical systems.
\newblock J. Symbolic Comput., 44\penalty0 (11):\penalty0 1551--1565.

\bibitem[Enciso(2016)]{Enciso2016}
Enciso G (2016)
\newblock Transient absolute robustness in stochastic biochemical networks.
\newblock J. R. Soc. Interface, 13\penalty0 (121):\penalty0 20160475.

\bibitem[Feinberg(1972)]{feinberg1972complex}
Feinberg M (1972)
\newblock Complex balancing in general kinetic systems.
\newblock Arch. Ration. Mech. Anal., 49\penalty0 (3):\penalty0 187--194.

\bibitem[Feinberg(1979)]{Feinberg1979lectures}
Feinberg M (1979)
\newblock Lectures on chemical reaction networks. {N}otes of lectures given at
  the {M}athematics {R}esearch {C}entre, {U}niversity of {W}isconsin. https://crnt.osu.edu/LecturesOnReactionNetworks.

\bibitem[Feinberg(1987)]{feinberg1987chemical}
Feinberg M (1987)
\newblock Chemical reaction network structure and the stability of complex
  isothermal reactors—I. the deficiency zero and deficiency one theorems.
\newblock Chem. Eng. Sci., 42\penalty0 (10):\penalty0 2229--2268.

\bibitem[Feinberg(1988)]{feinberg1988chemical}
Feinberg M (1988)
\newblock Chemical reaction network structure and the stability of complex
  isothermal reactors: {II.} multiple steady states for networks of deficiency
  one.
\newblock Chem. Eng. Sci., 43\penalty0 (1):\penalty0 1--25.

\bibitem[Feinberg(1995{\natexlab{a}})]{feinberg1995existence}
Feinberg M (1995{\natexlab{a}})
\newblock The existence and uniqueness of steady states for a class of chemical
  reaction networks.
\newblock Arch. Ration. Mech. Anal., 132:\penalty0 311--370.

\bibitem[Feinberg(1995{\natexlab{b}})]{feinberg1995multiple}
Feinberg M (1995{\natexlab{b}})
\newblock Multiple steady states for chemical reaction networks of deficiency
  one.
\newblock Arch. Rational Mech. Anal., 132:\penalty0 371--406.

\bibitem[Horn(1972)]{horn1972necessary}
Horn F (1972)
\newblock Necessary and sufficient conditions for complex balancing in chemical
  kinetics.
\newblock Arch. Ration. Mech. Anal., 49\penalty0 (3):\penalty0 172--186.

\bibitem[Horn and Jackson(1972)]{horn1972general}
Horn F, Jackson R (1972)
\newblock General mass action kinetics.
\newblock Arch. Ration. Mech. Anal., 47\penalty0 (2):\penalty0 81--116.

\bibitem[Johnston(2014)]{johnston2014translated}
Johnston MD (2014)
\newblock Translated chemical reaction networks.
\newblock Bulletin of mathematical biology, 76\penalty0 (5):\penalty0 1081--1116.

\bibitem[Johnston(2015)]{johnston2015computational}
Johnston MD (2015)
\newblock A computational approach to steady state correspondence of regular
  and generalized mass action systems.
\newblock Bull. Math. Biol., 77\penalty0 (6):\penalty0 1065--1100.

\bibitem[Karp(2012)]{karp2012complex}
Karp R L, Mill{\'a}n M P, Dasgupta T, Dickenstein A, Gunawardena J (2012)
\newblock Complex-linear invariants of biochemical networks.
\newblock J. Theoret. Biol., 311:\penalty0 130--138.

\bibitem[Mill{\'a}n et~al.(2012)Mill{\'a}n, Dickenstein, Shiu, and
  Conradi]{millan2012chemical}
Mill{\'a}n MP, Dickenstein A, Shiu A, Conradi C (2012)
\newblock Chemical reaction systems with toric steady states.
\newblock Bull. Math. Biol., 74\penalty0 (5):\penalty0 1027--1065.

\bibitem[M{\"u}ller and Regensburger(2012)]{muller2012generalized}
M{\"u}ller S, Regensburger G (2012)
\newblock Generalized mass action systems: Complex balancing equilibria and
  sign vectors of the stoichiometric and kinetic-order subspaces.
\newblock SIAM J. Appl. Math., 72\penalty0 (6):\penalty0 1926--1947.

\bibitem[M{\"u}ller et~al.(2016)M{\"u}ller, Feliu, Regensburger, Conradi, Shiu,
  and Dickenstein]{muller2016sign}
M{\"u}ller S, Feliu E, Regensburger G, Conradi C, Shiu A, Dickenstein A (2016)
\newblock Sign conditions for injectivity of generalized polynomial maps with
  applications to chemical reaction networks and real algebraic geometry.
\newblock Found. Comput. Math., 16\penalty0 (1):\penalty0 69--97.

\bibitem[Neigenfind et~al.(2013)Neigenfind, Grimbs, and
  Nikoloski]{neigenfind2013relation}
Neigenfind J, Grimbs S, Nikoloski Z (2013)
\newblock On the relation between reactions and complexes of (bio) chemical
  reaction networks.
\newblock J. Theoret. Biol., 317:\penalty0 359--365.

\bibitem[Shinar and Feinberg(2010)]{ShinarFeinberg2010Robustness}
Shinar G, Feinberg M (2010)
\newblock Structural sources of robustness in biochemical reaction networks.
\newblock Science, 327\penalty0 (5971):\penalty0 1389--1391.

\bibitem[Shinar et~al.(2007)Shinar, Milo, Martinez, and Alon]{Shinar2007}
Shinar G, Milo R, Martinez MR, Alon U (2007)
\newblock Input-output robustness in simple bacterial signaling systems.
\newblock Proc. Natl. Acad. Sci., 104:\penalty0 19931--19935.

\bibitem[Shinar et~al.(2009)Shinar, Rabinowitz, and Alon]{Shinar2009-2}
Shinar G, Rabinowitz JD, Alon U (2009)
\newblock Robustness in glyoxylate bypass regulation.
\newblock PLoS Comput. Biol., 5\penalty0 (3):\penalty0 e1000297.

\end{thebibliography}
\end{document}